\documentclass{amsart}
\usepackage{amscd,amssymb,amsopn,amsmath,amsthm,graphics,amsfonts,accents,enumerate,verbatim,calc}
\usepackage[dvips]{graphicx}
\usepackage[colorlinks=true,linkcolor=red,citecolor=blue]{hyperref}
\usepackage[all]{xy}

\usepackage{tikz}

\usepackage{tikz-cd}
\calclayout
\newcommand{\rt}{\rightarrow}
\newcommand{\lrt}{\longrightarrow}

\newcommand{\st}{\stackrel}

\newcommand{\la}{\lambda}
\newcommand{\La}{\Lambda}

\newcommand{\D}{\mathbb{D} }
\newcommand{\K}{\mathbb{K} }

\newcommand{\Z}{\mathbb{Z} }

\newcommand{\CA}{\mathcal{A} }

\newcommand{\CM}{\mathcal{M} }

\newcommand{\CP}{\mathcal{P} }
\newcommand{\CQ}{\mathcal{Q} }

\newcommand{\CS}{\mathcal{S} }
\newcommand{\CT}{\mathcal{T} }
\newcommand{\CX}{\mathcal{X} }
\newcommand{\CY}{\mathcal{Y} }

\newcommand{\Mod}{{\rm{Mod\mbox{-}}}}

\newcommand{\mmod}{{\rm{{mod\mbox{-}}}}}

\newcommand{\Hom}{{\rm{Hom}}}

\theoremstyle{plain}
\newtheorem{theorem}{Theorem}[section]
\newtheorem{corollary}[theorem]{Corollary}
\newtheorem{lemma}[theorem]{Lemma}

\newtheorem{proposition}[theorem]{Proposition}

\theoremstyle{definition}
\newtheorem{definition}[theorem]{Definition}
\newtheorem{example}[theorem]{Example}

\newtheorem{construction}[theorem]{Construction}

\newtheorem{remark}[theorem]{Remark}

\theoremstyle{plain}

\theoremstyle{definition}

\numberwithin{equation}{section}

\begin{document}

\title[relative Singularity categories  and singular equivalences]{relative Singularity categories and singular equivalences}

\author[Rasool Hafezi]{Rasool Hafezi }

\address{School of Mathematics, Institute for Research in Fundamental Sciences (IPM), P.O.Box: 19395-5746, Tehran, Iran}
\email{hafezi@ipm.ir}




\begin{abstract}
Let $R$ be a right notherian ring. We introduce the concept of relative singularity category $\Delta_{\CX}(R)$ of $R$ with respect to a contravariantly finite subcategory $\CX$ of $\mmod R.$ Along with some finiteness conditions on $\CX$, we prove that $\Delta_{\CX}(R)$ is triangle equivalent to a subcategory of  the homotopy category $\K_{\rm{ac}}(\CX)$ of exact complexes over $\CX$. As an application, a new description of the classical singularity category $\D_{\rm{sg}}(R)$ is given. The relative singularity categories are applied to lift a stable equivalence between two suitable subcategories of the module categories of two given right notherian rings to get a singular equivalence between the rings.  In different types of rings, including path rings, triangular matrix rings, trivial extension rings and tensor rings, we provide some consequences for  their singularity categories.
\end{abstract}

\maketitle
\section{introduction}
Let $R$ be a right notherian ring and $\mmod R $ the category of finitely generated right $R$-modules. The notion of {\it singularity category} of $R$ is defined to be Verdier quotient
$\D_{\rm{sg}}(R):=\frac{\D^{\rm{b}}(\mmod R)}{\K^{\rm{b}}(\rm{prj}\mbox{-}R)}$, where $\D^{\rm{b}}(\mmod R)$ denotes the bounded derived category, and $\K^{\rm{b}}(\rm{prj}\mbox{-}R)$ the homotopy category of bounded complexes whose terms are projective. This category was introduced by Buchweitz \cite{Bu} as a homological invariants of rings. If $R$ has finite global projective dimension, then every bounded complex
admits a finite projective resolution and therefore we have the equality $\D^{\rm{b}}(\mmod R)=\K^{\rm{b}}(\rm{prj}\mbox{-}R).$ In particular, the singularity category of a right notherian ring of finite global dimension is trivial. Hence, the singularity category provides a homological invariants for rings of infinite global dimension. The singularity  category is also used to detect some properties of a singularity. For example,  let $k$ be an algebraically closed field. A commutative complete Gorenstein
$k$-algebra $(R, m)$ satisfying $k=R/m$ has an isolated singularity if and only
if $\D_{\rm{sg}}(R)$ is a Hom-finite category, by work of Auslander \cite{A3}. Recently, the singularity category
 was applied by Orlove to study Landau-Ginzburg modules \cite{Or}.\\
Several relative versions of the singularity category are defined in the literatures by the different authors, for example see the papers \cite{W}, \cite{C5}, \cite{KY} and \cite{YZ2}. What we mean of the relative singularity categories is a special case of the relative ones given in \cite{W, C5}, and a generalization of the one studied in \cite{KY}. The relative concept of the singularity category here  is defined with  respect to a contravariantly finite subcategory of $\mmod R$. To be precise, let $\CX$ be a contravariantly finite subcategory of $\mmod R$ which contains $\rm{prj}\mbox{-}R.$ The relative singularity category of $R$ with respect to the subcategory $\CX$ is defined to be the Verdier quotient 
$$\Delta_{\CX}(R):= \frac{\D^{\rm{b}}(\mmod \CX)}{\CP},$$
where $\mmod \CX$ is the category of finitely presented contravariant functors from $\CX$ to the category of abelian groups, and $\CP$  the thick subcategory of the bounded derived category $\D^{\rm{b}}(\mmod \CX)$ generated by all functors $\Hom_R(-, P)\mid_{\CX},$ $P \in \rm{prj}\mbox{-}R$ (Definition \ref{Deefiniteion 1}). In the terminology  of \cite{AI} or \cite{IY}, the $\Delta_{\CX}(R)$ is only the {\it silting reduction} of $\D^{\rm{b}}(\mmod \CX)$ with respect to presilting subcategory $\CP$. It is shown in Theorem \ref{Theorem 3.1} the classical singularity category $\D_{\rm{sg}}(R)$ is a Verdier quotient of the relative singularity categories $\Delta_{\CX}(R)$. Hence due to this connection we try in this paper to use relative singularity categories to study the classical ones. The best cases of the subcategories $\CX$ might be more helpful are when the relative global $\CX$-dimensions with  respect to them are finite. Under such finiteness condition over $\CX$, the relative singularity category $\Delta_{\CX}(R)$ will become triangle equivalent to the Verdier quotient  $\frac{\K^{\rm{b}}(\CX)}{\K^{\rm{b}}(\rm{prj}\mbox{-}R)}$. This quotient category is studied independently, and   proved that it is triangle  equivalent to the homotopy category $\K^{-,\rm{p}}(\CX)$, see Proposition \ref{proposition 3.7}.
 Here $\K^{-,\rm{p}}(\CX)$ denotes the subcategory of the homotopy category $\K(\CX)$ of complexes over $\CX$ consisting of all  complexes which are homotopy-equivalent to a upper bounded exact  complex $\mathbf{P}$ over  $\CX$ such that for some $n,$ $\mathbf{P}^i \in  \rm{prj}\mbox{-}R$ for all $i\leqslant n.$ In fact, there are only finitely many terms of $\mathbf{P}$ to be non-projective. As a consequence we get the following theorem.
\begin{theorem}(Theorem \ref{Theorem 3.10})\label{Theorem 1.1} Let $\rm{prj}\mbox{-}R\subseteq \CX \subseteq \mmod R$ be a contravariantly finite subcategory. 	Assume the  global $\CX$-dimension of $R$ is finite. Then, there exists the following equivalence of  triangulated categories
	$$\mathbb{D}_{\rm{sg}}(R)\simeq \frac{\displaystyle \mathbb{K}^{-, \rm{p}}(\CX)}{ \displaystyle \mathbb{K}^{\rm{b}}_{\rm{ac}}(\CX)}.$$
\end{theorem}
As the above theorem says we obtain a description of the singularity category of $R$ such that only exact complexes are involved.\\
For right notherian rings $R$ and $R'$, we say that $R$ is {\it singularly equivalent} to $R'$ if there exists a triangle equivalence $\D_{\rm{sg}}(R)\simeq \D_{\rm{sg}}(R')$. Our next purpose is to apply the relative singularity categories to construct singular equivalences. Our attempts of using this approach lead to the following result.
\begin{theorem}(Theorem \ref{Coroallry 5.7})\label{Theorem 1.2}
	Let $\rm{prj}\mbox{-}R \subseteq \CX \subseteq \mmod R$ and $ \rm{prj}\mbox{-}R'\subseteq \CX' \subseteq \mmod R'$ be  contravariantly finite subcategories   and  to be closed under syzygies. Assume  the global $\CX$-dimension, resp. $\CX'$-dimension,  of $R$, resp. $R'$, is finite. Suppose, further,  there is a functor $F:\CX\rt \CX'$ such that $F(\rm{prj}\mbox{-}R) \subseteq \rm{prj}\mbox{-}R'$, the induced functor $\underline{F}:\underline{\CX}\rt \underline{\CX}'$ is an equivalence, and  for any $X \in \CX$, $F(\Omega_R(X))\simeq \Omega_{R'}(F(X))$ in $\underline{\CX}'$. If either of the following situations happens.
	\begin{itemize}
		\item [$(1)$]	
		The subcategories  $\CX$ and $\CX'$ satisfy the condition $(*)$ (see Definition \ref{Definition 4.1}). 		
		\item [$(2)$] There are quasi-resolving subcategories $\CX \subseteq \CY \subseteq \mmod R$ and $\CX' \subseteq \CY'\subseteq \mmod R'$  such that the functor $F$ is a restriction of an exact functor from  $\CY $ to $\CY'$.
	\end{itemize} 
	Then, $R$ and $R'$ are singularly equivalent.
\end{theorem}
 A subcategory of an abelian category is said to be {\it quasi-resolving} if it contains all projective objects, closed under direct summands, closed under kernels of epimorphisms. An exact functor from  $\CY$ to $\CY'$ means that any short exact sequence in $\mmod R$ with all terms in $\CY$ is mapped by the functor  into a short exact sequence in $\mmod R'$.\\  
By the above theorem we observer some sort of relative stable equivalence implies the singular equivalences. It is sometime easier to use the stable categories (which have more simple structure) to determine two rings to be singular equivalent. An important examples of the subcategories in the theorem is when both are the whole of the module categories. In this special case we are dealing with the stable equivalences which are fundamental equivalences both in the representation theory of algebras and groups
and in the theory of triangulated categories. We refer to Corollary \ref{Coroallry 5.7} and Remark \ref{lastremark} for a discussion about this special case.  Lifting stable equivalences to derived equivalences, and so singular equivalences, are also considered in \cite{As, D}, in the setting of self-injective algebras, and for more general algebras in \cite{HX}. We will provide some more examples of the subcategories appeared in the above theorems (not necessarily to be the whole of the module categories) to support our main results. Our examples are given in different types of rings, including path rings, triangular matrix rings, trivial extension rings and tensor rings. \\
 The paper is organized as follows. In Section $2$,  we  first collect some facts about our functorial approach. Then we  introduce  the notion of relative singularity categories  and whose connection with the usual singularity categories. In Section $3$, Theorem \ref{Theorem 1.1} is proved, and moreover, the Hom-finiteness of the relative singularity categories over Artin algebras is discussed. Section $4$ is devoted to prove Theorem \ref{Theorem 1.2} and based on this theorem to show that how one can lift a relative stable equivalences to a singular equivalence. In the last section we will  give some examples and application of our results.\\
{\bf Notation and  	Convention.} Throughout the paper $R$ denotes a right notherian ring and with $\Mod R$, resp. $\mmod R$, the category of, resp. finitely generated,  modules over $R$. By a module we always mean a right module unless otherwise stated.  The subcategory of finitely generated projective modules over $R$ is denoted by $\rm{prj}\mbox{-}R$. Let $\CA$ be an additive category. We denote by $\mathbb{K}(A)$ the homotopy category of all complexes over $\CA$. Moreover,  $\mathbb{K}^{\rm{b}}(\CA)$ denotes the full subcategory of $\mathbb{K}(\CA)$ consisting of all bounded above, resp. bounded, complexes. In case that $\CA$ is abelian, the derived category of $\CA$ will be denoted by $\mathbb{D}(\CA)$, which is the  Verdier quotient  $\mathbb{K}(\CA)/\mathbb{K}_{\rm{ac}}(\CA)$. Here $\mathbb{K}_{\rm{ac}}(\CA)$ is the  homotopy category of all exact complexes over $\CA$. By   $\mathbb{D}^{\rm{b}}(\CA)$, we denote the full subcategory of $\mathbb{D}(\CA)$ consisting of all  homologically bounded complexes. All subcategories are assumed to be {\it  full,  closed under isomorphisms, direct summands and finite sums}.  Hom-spaces in the homotopy category with ``$\Hom_{\mathbb{K}}(-, -)$'' is usually shown. For an $R$-module $M$ in $\mmod R$, consider a short exact sequence $0 \rt \Omega_R(M) \rt P\rt M \rt 0$ with $P$ in $\rm{prj}\mbox{-}R$.  The module $\Omega_R(M)$ is then called a syzygy module of $M$. Note that syzygy modules
of M are not uniquely determined. An $n$-th syzygy of M will be denoted by $\Omega^n_R(M)$, for $n\geqslant 2.$ For a complex  $\mathbf{X}$ in $\K(\CX)$, we have in mind the differentials raise the degrees. If $i>0$, resp. $i<0,$ we mean by $\mathbf{X}[i]$ the shift of the complex  $\mathbf{X}$ to the left, resp, right,  $i$, resp. $-i$, degrees. A module is considered as a complex concentrated at degree zero when we want to assume it as an object in the category of complexes. Let $\mathcal{C}$ be a subcategory of an abelian category $\CA$. Given an object $M$ in $\CA$,
a right $\CX$-approximation of $M$ is a map $g:C \rt M$ with $C \in \mathcal{C}$
such that for any map $h:C'\rt M$ with $C'\in \mathcal{C}$, there is a map $ f:C'\rt C$
 such that $h=g \circ f.$ In case every object in $\CA$ has a right
$\mathcal{C}$-approximation, $\mathcal{C}$ is said to be {\it  contravariantly finite} in $\CA$. Let $\rm{prj}\mbox{-}R \subseteq \CX$ be a subcategory of $\mmod R$. The stable category of $\CX$ is denoted by $\underline{\CX}$. The stable  category is defined by setting the objects to be the same as  those of $\CX$, and for any $X$ and $Y$ in $\CX$, the group of  morphisms is given by  $\underline{\rm{Hom}}_R(X,Y)=\Hom_R(X, Y)/\rm{P}\Hom_R(X, Y)$, where $\rm{P}\Hom_R(X, Y)$ denotes the set of all morphisms $A \rt B$ which factor through a projective module. 
\section{Relative Singularity Categories}
Throughout  this section unless stated otherwise, let $\CX$ be a  contravariantly finite subcategory of $\mmod R$ containing $\rm{prj}\mbox{-}R.$ As a general point in our paper, for avoiding any confusion, at least in our main results we often restated the needed assumptions. In this section, we will introduce the notion of relative singularity category $\Delta_{R}(\CX)$ with respect to the subcategory $\CX$ as a Verdier Localization of $\D^{\rm{b}}(\mmod \CX)$. Then, we intend  to make a realization of $\D_{\rm{sg}}(R)$ by a Verdier localization of $\Delta_{R}(\CX)$. To do this, we need first to recall some functorial construction given in \cite{HK}.\\
Let $F$ be in $\mmod \CX$ and $\Hom_R(-, X_1)\mid_{\CX}\st{\Hom_R(-, d)\mid_{\CX}}\lrt \Hom_{R}(-, X_0)\mid_{\CX}\rt F \rt 0$ a projective presentation of $F$. The assumption of $\CX$ being contravariantly finite implies that any morphism in $\CX$ has a weak kernel, consequently $\mmod \CX$ is an abelian category. The functor $\vartheta: \mmod \CX\rt \mmod R$ is defined by sending $F$ to the cokernel $\rm{Cok} \ d$ of $d$ in $\mmod R.$ The assignment $\vartheta$ on morphisms is defined with help of the lifting property, see \cite[Remark 2.1]{HK} for more details. The functor $\vartheta$ is an exact functor. Therefore, it induces a triangle functor $\D^{\rm{b}}_{\vartheta}$ from $\D^{\rm{b}}(\mmod \CX)$ to $ \D^{b}(\mmod R)$. It acts on objects, as well as roofs, terms by terms. Let $\D^{\rm{b}}_0(\mmod \CX)$ denote the kernel of $\D^{\rm{b}}_{\vartheta}$, by \cite[Proposition 3.1]{HK}, it consists of all complexes $\mathbf{K}$ such that the associated valuated complexes $\mathbf{K}(P)$ on any projective module $P$ in $\rm{prj}\mbox{-}R$ are an exact complex of abelian groups. The induced functor from the Verdier localization $\frac{\D^{\rm{b}}(\mmod \CX)}{\D^{\rm{b}}_0(\mmod\CX)}$ to $\D^{\rm{b}}(\mmod R)$ will be denoted by $\widetilde{\D^{b}_{\vartheta}}$. In \cite[Proposition 3.3]{HK}, it is proved that $\widetilde{\D^{b}_{\vartheta}}$ is an 
 equivalence of triangulated categories. By $\CP$ we will show the thick subcategory of $\D^{\rm{b}}(\mmod \CX)$ generated by all representable functors $\Hom_R(-, P)\mid_{\CX}$ in which $P$ runs through modules in $\rm{prj}\mbox{-}R \subseteq \CX.$ As for any $X \in \CP$ and $Y \in \D_0^{\rm{b}}(\mmod \CX)$, $\Hom_{\D^{\rm{b}}(\mmod \CX)}(X, Y)=0$, then $\CP$ can be considered as a thick subcategory of $\frac{\D^{\rm{b}}(\mmod \CX)}{\D^{\rm{b}}_0(\mmod\CX)}$. Moreover, by the same reason, the triangulated category $\D^{\rm{b}}_0(\mmod\CX)$ can be identified as a subcategory of $\D^{\rm{b}}(\mmod \CX)/\CP$. For the latter embedding we need to  define morphisms in the quotient category  with right roofs.  
 \begin{definition}\label{Deefiniteion 1}
 	The {\it relative singularity category} with  respect to a contravariantly finite subcategory $\rm{prj}\mbox{-} R \subseteq \CX \subseteq \mmod R$  is the Verdier quotient category
 	$$\Delta_{\CX}(R):= \frac{\D^{\rm{b}}(\mmod \CX)}{\CP}.$$	
 \end{definition}
When $\CX=\rm{prj}\mbox{-}R$, then $\Delta_{\rm{prj}\mbox{-}R}(R)$ is not nothing else that than the usual singularity category $\D_{\rm{sg}}(R)$. For the  special case, $\CX=\rm{Gprj}\mbox{-}R$, the resulting singularity category is called the {\it Gorenstein singularity category}, assuming $\rm{Gprj}\mbox{-}R$ is contravariantly finite in $\mmod R$. For instance, over {\it virtually Gorenstein Artin algebras} the subcategory of Gorenstien projective modules is always contravariantly finite. This  class of algebras which has been introduced in \cite{BR}.\\ 
In the next result  a connection between the classic singularity category and the relative one is stated.
 \begin{theorem}\label{Theorem 3.1}
 	Let $\CX$ be a  contravariantly finite subcategory of $\mmod R$ containing $\rm{prj}\mbox{-}R.$ Then,  we have the following equivalences of triangulated categories
 	$$ \D_{\rm{sg}}(R)\simeq\frac{\D^{\rm{b}}(\mmod \CX)/\D^{\rm{b}}_0(\mmod\CX)}{\CP}\simeq \frac{\D^{\rm{b}}(\mmod \CX)}{\rm{Thick}(\D^{\rm{b}}_0(\mmod\CX)\cup \CP)}\simeq \frac{\D^{\rm{b}}(\mmod \CX)/\CP}{\D^{\rm{b}}_0(\mmod\CX)}.$$
 	In particular, $\D_{\rm{sg}}(R)\simeq \frac{\Delta_{\CX}(R)}{\D^{\rm{b}}(\mmod \CX)}.$
 \end{theorem}
\begin{proof}
	For the first equivalence, note that by the definition one can see that $\vartheta(\Hom_R(-, P)\mid_{\CX})=P$ for any $P \in \rm{prj}\mbox{-}R,$ but this fact implies that the functor $\widetilde{\D^{b}_{\vartheta}}$ can be restricted to the subcategories $\CP$ and $\K^{\rm{b}}(\rm{prj}\mbox{-}R)$. Hence, $\widetilde{\D^{b}_{\vartheta}}$ induces a triangle equivalence between the Verdier localization categories $\frac{\D^{\rm{b}}(\mmod \CX)/\D^{\rm{b}}_0(\mmod\CX)}{\CP}$  and $\D^{\rm{b}}(\mmod R)/ \K^{\rm{b}}(\rm{prj}\mbox{-}R)$, as desired. The last two equivalences follow
	from the universal property of the triangulated quotient categories,  see also \cite[Proposition 2.3]{O}. 
\end{proof}

 If the ring $R$ is clear from the context, we often use $\Delta(\CX)$ instead of $\Delta_{\CX}(R)$, e.g., $\Delta(\rm{Gprj}\mbox{-}R)$ instead of the long notation  $\Delta_{(\rm{Gprj}\mbox{-}R)}(R).$ Hence, by our theorem, over a virtually Gorenstein Artin algebra $\La,$ we can describe $\D_{\rm{sg}}(\La)$ as the Verdier quotient $\frac{\Delta(\rm{Gprj}\mbox{-}\La)}{\D^{\rm{b}}_0(\mmod (\rm{Gprj}\mbox{-}\La))}$.

By putting some finiteness conditions on $\CX$ as in the following,  we will obtain more interesting  description of $\D_{\rm{sg}}(R)$.
For a module $M$ in $\mmod R,$ we say that the {\it  $\CX$-dimension} of $M$ is finite, if there an exact sequence $0 \rt X_n\rt \cdots\rt X_1\rt X_0\rt M \rt 0$ in $\mmod R$ with all $X_i \in \CX$ and it remains exact by applying $\Hom_{R}(X, -)$ for each $X \in \CX.$ We say that the {\it global $\CX$-dimension of $R$} is finite if any module  in $\mmod R$ has of finite $\CX$-dimension.
\begin{lemma}\label{Lemma3.2}
	Assume that  the global  $\CX$-dimension of $R$ is finite.
	 Then, the Yoneda functor induces the following triangle equivalence
	$$\D^{\rm{b}}(\mmod \CX)\simeq \K^{\rm{b}}(\CX).$$	
\end{lemma} 
\begin{proof}
Take $F$ in $\mmod \CX$ with a projective presentation $$\Hom_{R}(-, X_1)\mid_{\CX}\st{\Hom_{R}(-, d)\mid_{\CX}}\lrt\Hom_{R}(-, X_0)\mid_{\CX}\rt F\rt 0.$$ By our assumption, there is an exact sequence 
$$0 \rt X_n\rt \cdots\rt X_2\rt \rm{Ker} \ d \rt 0$$	
with $X_i$ belong to $\CX$ and remains exact by applying $\Hom_R(X, -)$ for any $X \in \CX.$ By applying the Yoneda functor on the above exact sequence and then gluing the obtained exact sequence in the functor category  with  $0 \rt\Hom_{R}(-, \rm{Ker} \ d)\mid_{\CX}\rt \Hom_{R}(-, X_1)\mid_{\CX}\rt \Hom_{R}(-, X_0)\mid_{\CX}\rt F\rt0$, obtained from the projective presentation, we observe the projective dimension of $F$ in $\mmod \CX$ is finite. On the other hand, it is known that in this case the existence of  the triangle equivalence $\D^{\rm{b}}(\mmod \CX)\simeq \K^{\rm{b}}(\rm{prj}\mbox{-}\CX)$, where $\rm{prj}\mbox{-}\CX$ denotes the category of projective functors in $\mmod \CX$. By our convention since $\CX$ is closed under direct summands then $\rm{prj}\mbox{-}\CX$ consists of all representable functors $\Hom_{R}(-, X)\mid_{\CX}$, where $X \in \CX$. Using this fact and the Yoneda functor turns out the triangle equivalence $\K^{\rm{b}}(\rm{prj}\mbox{-}\CX)\simeq \K^{\rm{b}}(\CX)$. 
So we are done.	
\end{proof}
The equivalence in the above lemma induces the triangle equivalence $\Delta_{\CX}(R)\simeq \frac{\K^{\rm{b}}(\CX)}{\K^{\rm{b}}(\rm{prj}\mbox{-}R)},$ see the proof of the next theorem for more explanations.
 \begin{theorem}\label{Theorem 3.3}  Let $\CX$ be a  contravariantly finite subcategory of $\mmod R$ containing $\rm{prj}\mbox{-}R.$
 	Assume that  the global $\CX$-dimension of $R$ is finite. Then, we have the following equivalences of triangulated categories
 	$$ \D_{\rm{sg}}(R)\simeq\frac{\K^{\rm{b}}( \CX)/\K^{\rm{b
 		}}_{\rm{ac}}(\CX)}{\K^{\rm{b}}(\rm{prj}\mbox{-}R)}\simeq \frac{\K^{\rm{b}}( \CX)}{\rm{Thick}(\K^{\rm{b}}_{\rm{ac}}(\CX)\cup \K^{\rm{b}}(\rm{prj}\mbox{-}R))}\simeq \frac{\K^{\rm{b}}(\CX)/\K^{\rm{b}}(\rm{prj}\mbox{-}R)}{\K^{\rm{b}}_{\rm{ac}}(\CX)}.$$
 \end{theorem}
   \begin{proof}
   	Let us first give some  more information for the equivalence $\D^{\rm{b}}(\mmod \CX)\simeq \K^{\rm{b}}(\CX)$ given in Lemma \ref{Lemma3.2} which is helpful to obtain the equivalences in the statement. The equivalence on objects acts as follows. Let $\mathbf{X}$ be a complex in $\D^{\rm{b}}(\mmod \CX)$ and   $\Theta:\mathbf{P}\rt \mathbf{X}$ a $\mathbb{K}$-projective resolution of $\mathbf{X}$. By our assumption the $\mathbb{K}$-projective resolution $\mathbf{P}$ has to be a bounded complex of projective functors. Because of the Yoneda lemma, $\mathbf{P}$ can be presented as a complex obtaining by applying the Yoneda functor on some bounded complex of modules in $\CX$, say $\mathbf{Q}$. In fact, under the equivalence $\mathbf{X}$ is mapped into  a complex in $\K^{\rm{b}}(\CX)$ homotopy equivalent to $\mathbf{Q}$. Therefore, by the construction of the equivalence as already explained, one can see the equivalence can be restricted to $\D^{\rm{b}}_0(\mmod \CX)\simeq \K^{\rm{b}}_{\rm{ac}}(\CX)$ and $\CP\simeq \K^{\rm{b}}(\rm{prj}\mbox{-}\La)$. For the first  restricted equivalence, assume $\mathbf{X}$ belongs to $\D^{\rm{b}}_0(\mmod \CX)$,   we consider the triangle $\mathbf{P}\rt \mathbf{X}\rt \rm{cone} (\Theta)\rt \mathbf{P}[1]$, as described in the above. The valuation of the triangle on the regular module $R \in \rm{prj}\mbox{-}R$  gives us the triangle $\mathbf{P}(R)\rt \mathbf{X}(R)\rt \rm{cone} (\Theta)(R)\rt \mathbf{P}(R)[1]$ in the bounded derived category of abelian groups. By the valuated triangle we observe that $\mathbf{P}(R)$ is an exact complex of abelian groups. The isomorphism $\mathbf{P}(R)=\Hom_{R}(R, \mathbf{Q})\simeq \mathbf{Q}$, as complexes of $R$-modules, yields $\mathbf{Q}$ is an exact complex, so $\mathbf{Q}$ is in $\K^{\rm{b}}_{\rm{ac}}(\CX)$, as desired. Now applying Theorem \ref{Theorem 3.1} 
   	and using the above-mentioned observations we can complete the proof.
   \end{proof}
Fortunately, over an arbitrary Artin algebra $\La$ there is always a subcategory $\CX \subseteq \mmod \La$  which satisfies the assumption we need in  Theorem \ref{Theorem 3.3}. For instance, Auslander in  his Queen Mary College lecture \cite{A} showed the global dimension of $\rm{End}(\oplus_{i=0}^n\La/J^i)$, where $J$ is radical  and $n$ lowey length of $\La$, respectively, is finite. Consequently, $\CX=\rm{add}\mbox{-}\oplus_{i=0}^n\La/J^i$ holds the needed conditions of Theorem \ref{Theorem 3.3}. Therefore, we can always estimate the singularity categories of  Artin algebras via an Artin  algebra of finite global dimension. 
 
Recall that a two-sided notherian ring $R$ is {\it Gorenstein} provided that $\rm{id}_R R<\infty$ and $\rm{id}_{R^{\rm{op}}}R <\infty.$ Let $M$ be in $\mmod R$ and $M^*=\Hom_R(M, R)$.
Recall that $M$ is {\it Gorenstein projective} provided that there is an exact complex $\mathbf{P}$ of finitely generated projective $R$-modules such that the Hom-complex $\mathbf{P}^*=\Hom_R(\mathbf{P}, R)$ is still exact and $M$ is isomorphic to a certain cocycle $Z^i(\mathbf{P})$ of $\mathbf{P}$. We denote by $\rm{Gprj}\mbox{-}R$ the subcategory of $\mmod R$ formed by all Gorenstein projective $R$-modules.
\begin{example}
	let $R$ be a Gorenstein ring. It is known that over Gorenstein rings any module has of finite Gorenstein dimension \cite{EJ}. So, the subcategory $\rm{Gprj}\mbox{-}R$ satisfies the assumption of Theorem \ref{Theorem 3.3} and so we can apply the theorem to get the equivalences $\D_{\rm{sg}}(R)\simeq \frac{\Delta(\rm{Gprj}\mbox{-}R)}{\K^{\rm{b}}_{\rm{ac}}(\rm{Gprj}\mbox{-}R)}.$ The subcategory $\rm{Gprj}\mbox{-}R$ inherits an exact structure from $\mmod R$, which is Frobenius. Then by the Happel's work the stable category $\underline{\rm{Gprj}}\mbox{-}R$ becomes naturally a triangulated category. We know by  Buchweitz's theorem \cite{Bu} the existence of a triangle equivalence $\D_{\rm{sg}}(R)\simeq \underline{\rm{Gprj}}\mbox{-}R.$ Now by combing the equivalences we get the triangle equivalence $\underline{\rm{Gprj}}\mbox{-}R\simeq \frac{\Delta(\rm{Gprj}\mbox{-}R)}{\K^{\rm{b}}_{\rm{ac}}(\rm{Gprj}\mbox{-}R)}.$
\end{example}

Let us summarize the equivalences of the triangulated categories appeared above in the following diagram. To have the equivalences on the left side we need to assume that the global $\CX$-dimension of $R$ is finite.
\begin{equation*}\label{E:importantdiagram}
\begin{array}{c}
\begin{xy}\SelectTips{cm}{}
\xymatrix{
	\frac{ \displaystyle \Delta_{\CX}(R)}{\K^{\rm{b}}_{\rm{ac}}(\CX)}\simeq \frac{\displaystyle \K^{\rm{b}}(\CX)/\K^{\rm{b}}(\rm{prj}\mbox{-}R)}{\displaystyle \K^{\rm{b}}_{\rm{ac}}(\CX)} & & \\
\frac{\displaystyle \K^{\rm{b}}(\CX)/\K^{\rm{b}}_{\rm{ac}}(\CX)}{\displaystyle \K^{\rm{b}}(\rm{prj}\mbox{-}R)} \ar[r]^-{\sim} \ar[u]^-\sim &\frac{\displaystyle \D^{\rm{b}}(\mmod \CX)/\D^{\rm{b}}_{0}(\mmod \CX)}{\displaystyle \CP} \ar[r]^-\sim & \D_{\rm{sg}}(R)  \\
\frac{\displaystyle \K^{\rm{b}}(\CX)}{\displaystyle \K^{\rm{b}}_{\rm{ac}}(\CX)} \ar[r]^{\sim}\ar[u] & \frac{\displaystyle \D^{\rm{b}}(\mmod \CX)}{\displaystyle \D^{\rm{b}}_0(\mmod \CX)}\ar[u]^-{\mathsf{can}} \ar[r]^{\widetilde{\D^{\rm{b}}_{\vartheta}} } &  \D^{\rm{b}}(\mmod R)\ar[u] \\
\K^{\rm{b}}(\rm{prj}\mbox{-}R)\ar@{^{(}->}[u]\ar[r]^{\sim}	&\CP  \ar@{^{(}->}[u]\ar[r]^{\sim} &
\K^{\rm{b}}(\rm{prj}\mbox{-}R)\ar@{^{(}->}[u]
}\end{xy}
\end{array}
\end{equation*}

\section{ (relative) singularity categories via homotopy categories of exact complexes}
Throughout  this section,  let $\rm{prj}\mbox{-}R\subseteq \CX \subseteq \mmod R$. In this section we will study the Verdier quotient $\frac{\K^{\rm{b}}(\CX)}{\K^{\rm{b}}(\rm{prj}\mbox{-}R)}$ and show that it can be embedded into the homotopy category of exact complexes over $\CX.$ Then by applying our result from the preceding section in conjunction with the embedding we give a new description of $\D_{\rm{sg}}(R)$ such that only exact complexes are involved.

 Let $\K^{-,\rm{p}}(\CX)$ denote the subcategory of the homotopy category $\K(\CX)$ consisting of all  complexes which are homotopy-equivalent to a upper bounded exact  complex $\mathbf{P}$ over  $\CX$ such that for some $n,$ $\mathbf{P}^i \in  \rm{prj}\mbox{-}R$ for all $i\leqslant n.$ In fact, there are only finitely many terms of $\mathbf{P}$ to be non-projective. We know that the inclusion $\K_{\rm{ac}}(\Mod R)\hookrightarrow \K(\Mod R)$ has a right adjoint, denoted by $\Phi$. The right adjoint sends a  complex $\mathbf{X}$ to the complex $\mathbf{Y}$ fitting into a triangle $P_{\mathbf{X}}\rt \mathbf{X}\rt \mathbf{Y}\rt \mathbf{P}[1]$, where $P_{\mathbf{X}}$ is a $\K$-projective complex, see \cite[Proposition 1.6]{AJS}.   Since  $\K^{\rm{b}}(\CX)$ is  generated by all complexes concentrated at degree zero with terms in $\CX$, then  the essential image (in $\K(\CX)$) of the  restricted functor $\Phi$ over $\K^{\rm{b}}(\CX)$  is exactly $\K^{-,\rm{p}}(\CX)$ (see the proof of the next result for more explanations). Since the restricted functor vanishes on $\K^{\rm{b}}(\rm{prj}\mbox{-}R)$,  hence by the universal property there is  a triangle functor $\Psi:  \frac{\K^{\rm{b}}(\CX)}{\K^{\rm{b}}(\rm{prj}\mbox{-}R)} \rt \K^{-,\rm{p}}(\CX)$.
To distinguish, we sometimes use $\Psi_{\CX}$ when we are dealing   with several subcategories, simultaneously. For any $X$ in $\CX$, we fix the following exact complex in $\K^{-,\rm{p}}(\CX)$,
$$\mathbf{C}_X: \cdots \rt P^n_X\st{d^n_X}\rt \cdots \rt P^1_X\st{d^1_X}\rt P^0_X\st{d^0_X}\rt X\rt 0\rt \cdots, $$
where $P^i_X$ are in $\rm{prj}\mbox{-}R,$ and $X$ at degree $0$ and $P^i_X$ at degree $-i-1.$ Let $\mathbf{P}_X$ denote the deleted projective resolution induced by $\mathbf{C}_X$, there is the triangle $\mathbf{P}_X\rt X\rt \mathbf{C}_X\rt \mathbf{P}_X[1]$. Hence by the triangle, we may assume  $\Phi(X)=\mathbf{C}_X.$ We will use this convention throughout  the paper.
\begin{proposition}\label{proposition 3.7}
	The  functor $\Psi$, defined in the above, 
	 gives the following equivalence of triangulated categories
	 $$\frac{\K^{\rm{b}}(\CX)}{\K^{\rm{b}}(\rm{prj}\mbox{-}R)}\simeq \K^{-,\rm{p}}(\CX).$$	 
\end{proposition}
\begin{proof} We first show that the essential image of $\Psi$, or equivalently $\Phi$,  is equal to $\K^{-,\rm{p}}(\CX)$. Hence $\Psi$ is dense. Since $\CX$ generates $\K^{\rm{b}}(\CX)$, meaning that $\K^{\rm{b}}(\CX)$ is the smallest  triangulated 
	subcategory contains $\CX,$ hence all $\Psi(X)=\mathbf{C}_X$ generates the essential image. Since any $\mathbf{C}_X$ belongs to $\K^{-,\rm{p}}(\CX)$, so the essential image of $\Psi$ is included in $\K^{-,\rm{p}}(\CX).$ 
 Conversely, Let $\mathbf{X}$ be a complex in  $\K^{-,\rm{p}}(\CX).$ We may present it as in the following 
$$\cdots \rt P^1\rt P^0\rt X^n\rt \cdots\rt X^1\rt X^0\rt 0$$
where $P^i$ in $\rm{prj}\mbox{-}R$. If all $X^i$, for $i\geqslant 1$ are projective, then $\mathbf{X}$ becomes a projective resolution of $X^0$. Hence in this case $\mathbf{X}\simeq \Psi(X^0)$, so it is in the essential image. Take a projective resolution of $X_0$
$$\cdots \rt Q^m\rt\cdots \rt Q^1\rt Q^0\rt X^0\rt 0.$$
By the lifting property we get the following chain map

\begin{equation*}
\xymatrix{\cdots\ar[r]&Q^m \ar[r]\ar[d] &\cdots\ar[r]&  Q^n \ar[r] \ar[d]&  \cdots \ar[r] &  Q^0 \ar[r] \ar[d]&  X^0\ar@{=}[d]\ar[r]&0\\ \cdots\ar[r]&P^n \ar[r]& \cdots\ar[r]& X^n \ar[r] & \cdots \ar[r] & X^1\ar[r] & X^0\ar[r]&0}
\end{equation*}
It is easy to see that the mapping cone of the above chain map is  homotopy equivalent to the following complex
$$\cdots \rt Q^n\oplus P^0\rt X^n\rt \cdots\rt Q^0\oplus X^2 \rt X^1\rt0. $$
 The obtained complex has non-projective terms in at most $n-1$ degrees. Repeating this argument we can reach to a  complex in  $\K^{-,\rm{p}}(\CX)$ with at most one non-projective term, which can be considered as   a shifting of a projective resolution of some object in $\CX$. It must be homotopy-equivalent to $\mathbf{C}_X[i]$ for some $i \in \Z$ and $X$ in $\CX.$ Therefore, by such a construction we can see $\mathbf{X}$  must be in the essential image, as by the construction it lies in a sequence of the triangles,  as described in Remark \ref{remark1}. Now we prove the functor $\Psi$ is full and faithful.  Recall that $\Psi$ acts on objects the same as $\Phi$, and for a right roof $\mathbf{X}\st{f}\rt \mathbf{Z} \st{g}\leftarrow \mathbf{Y}$, with $\rm{cone}(g)$ in $\K^{\rm{b}}(\rm{prj}\mbox{-}R)$, is defined by $\Phi(g)^{-1}\circ\Phi(f)$. Assume $\Phi$ becomes zero on a roof $\mathbf{X}\st{f}\rt \mathbf{Z} \st{g}\leftarrow \mathbf{Y}$. Then $\Phi(f)$ has to be zero. By the construction of $\Phi$ on morphisms there is the following commutative diagram with the  triangles on rows 
	$$\xymatrix{
	P_{\mathbf{X}} \ar[r] &\mathbf{X} \ar[d]^{f} \ar[r] & \Phi(\mathbf{X}) \ar[d]^{\Phi(f)}
	\ar[r] &P_{\mathbf{X}}[1] \\
	P_{\mathbf{Z}} \ar[r] & \mathbf{Z}  \ar[r] &\Phi(\mathbf{Z})
	\ar[r] &P_{\mathbf{Z}}[1].	} 	 $$
 Since $\Phi(f)$ is zero then $f$ factors through the morphism $P_{\mathbf{Z}}\rt \mathbf{Z}$. Since $\mathbf{X}$ and $\mathbf{Z}$ both are bounded complexes and $P_{\mathbf{Z}}$ a upper bounded complex, then we can deduce that the morphism $f$ factors through some  brutal truncation $P_{\mathbf{Z}}^{\geq m}$, that is in $\K^{\rm{b}}(\rm{prj}\mbox{-}\La)$. Hence $\Psi$ is faithful. For fullness, since $\CX$ generates $\K^{\rm{b}}(\CX)$, it suffices to  show that for any $X$ and $Y$ in $\CX$, and $i, j \in \mathbb{Z}$, the induced group homomorphism
$$\Hom_{\frac{\K^{\rm{b}}(\CX)}{\K^{\rm{b}}(\rm{prj}\mbox{-}\La)}}(X[i], Y[j])\rt \Hom_{\K^{-,\rm{p}}(\CX)}(\mathbf{C}_X[i], \mathbf{C}_Y[j])$$

is surjective. For simplicity, we only prove for $i=0.$ The exact complexes $\mathbf{C}_X$ and $\mathbf{C}_Y$, as our notations, are presented as below, respectively,
$$\mathbf{C}_X: \cdots \rt P^n_X\rt \cdots \rt P^1_X\rt P^0_X\rt X\rt 0\rt \cdots, $$
$$\mathbf{C}_Y: \cdots \rt P^n_Y\rt \cdots \rt P^1_Y\rt P^0_Y\rt Y\rt 0\rt \cdots, $$
where $X$ and $Y$ are settled in degree $0$. Let $[f]$ be a homotopy equivalence class in $\Hom_{\K^{-,\rm{p}}(\CX)}(X, Y[j])$. For $j>0,$ It is  easy to see that (only by using the lifting property)  $\Hom_{\K^{-,\rm{p}}(\CX)}(\mathbf{C}_X, \mathbf{C}_Y[j])=0$, so nothing to prove.   For $j=0$, the equivalence class of the right roof $X \st{f^0}\rt Y \st{1}\leftarrow Y$ is mapped into $[f]$ by $\Psi$.  It remains for the case  $j < 0$. The chain map $f$ is presented as below
\begin{equation*}
\xymatrix{\cdots\ar[r]&P^0_X \ar[r]\ar[d]^{f^{1}} & X \ar[r]\ar[d]^{f^0}&  0 &   &   &&\\ \cdots\ar[r]&P^{-j}_Y \ar[r]& P^{-j-1}_Y\ar[r]& \cdots \ar[r]&  P^0_Y \ar[r] & Y\ar[r] & 0&}
\end{equation*}

Put $\mathbf{Z}$ to be the complex $0 \rt P^{-j-1}_Y\rt \cdots \rt P^0_Y\rt Y\rt 0$, where $P^{-j-1}_Y$ is at degree $0$. Let $g:X\rt \mathbf{Z}$, resp. $s: Y[j] \rt \mathbf{Z},$ be a chain map to be zero on all degrees except degree $0$, resp. $ -j$,  with $f^0$, resp. $\rm{id_Y}$. We claim that the equivalence class of the right roof $X\st{[g]}\rt \mathbf{Z}\st{[s]}\leftarrow Y[j]$
is mapped into $[f]$.  Consider the following commutative diagram

\begin{equation*}
\xymatrix{\cdots\ar[r]&P^{-j+2}_Y \ar[r]\ar[d] &P^{-j+1}_Y\ar[r]\ar[d]&  P^{-j}_Y \ar[r] \ar[d]&  \cdots \ar[r] & 0 \ar[r] \ar[d]&  0\ar[d]\ar[r]&0\\ \cdots\ar[r]&0 \ar[d]\ar[r]& 0\ar[d]\ar[r]& P^{-j-1}_Y \ar@{=}[d]\ar[r] & \cdots \ar[r] & P^0_Y\ar@{=}[d]\ar[r] & Y\ar@{=}[d]\ar[r]&0\\\cdots \ar[r]&P^{-j+1}_Y\ar[r]&P^{-j}_Y\ar[r]&P^{-j-1}_Y\ar[r]&\cdots\ar[r]&P^0_Y\ar[r]&Y\ar[r]&0.}
\end{equation*}
 From the above diagram we have the following triangle 
 $$\mathbf{C}^{\leq j-1}_Y[j-1]\rt \mathbf{Z}\rt \mathbf{C}_Y[j]\rt \mathbf{C}^{\leq j-1}_Y[j]. $$
Since $\mathbf{C}^{\leq j-1}_Y[j-1]$ belongs to $\K^{-}(\rm{prj}\mbox{-}\La)$
and $\mathbf{C}_Y[j]$  an exact complex, then by the triangle we can conclude that  $\Phi(\mathbf{Z})\simeq \mathbf{C}_Y[j].$ In addition,  we can see the equalities $\Phi([g])=[f]$ and $\Phi([s])=\rm{id}_{\mathbf{C}_Y[j]}$,  only by the  definition of $\Phi$ on morphisms together with using the triangle. Hence $\Psi$ assigns the morphism $[f]$ to the equivalence class of the right  roof $X\st{[g]}\rt \mathbf{Z}\st{[s]}\leftarrow Y[j]$, what we wanted to prove.
 So the proof is now complete.
\end{proof}

\begin{remark}\label{remark1}
As a result from the first part of the proof of Proposition \ref{proposition 3.7}, we can see for any complex $\mathbf{X}$ in $\K^{-, \rm{p}}(\CX)$	there is a finite sequences  of triangles as the following
$$\mathbf{C}_{M_i}[r_i]\rt \mathbf{X}_i\rt \mathbf{X}_{i+1}\rt \mathbf{C}_{M_i}[r_i+1],$$
where $0\leqslant i \leqslant n$, $r_i \in \mathbb{Z}$ and $M_i \in \CX$, $\mathbf{X}_0=\mathbf{X}$ and $\mathbf{X}_{n+1}=0.$ 
\end{remark}

 As an application of the above proposition in the next result  we give a  nice  description for the usual singularity category $\D_{\rm{sg}}(R)$.
\begin{theorem}\label{Theorem 3.10} Let $\rm{prj}\mbox{-}R\subseteq \CX \subseteq \mmod R$ be a contravariantly finite subcategory. 	Assume the  global $\CX$-dimension of $R$ is finite. Then, there exists the following equivalence of  triangulated categories
	$$\mathbb{D}_{\rm{sg}}(R)\simeq \frac{\displaystyle \mathbb{K}^{-, \rm{p}}(\CX)}{ \displaystyle \mathbb{K}^{\rm{b}}_{\rm{ac}}(\CX)}.$$
\end{theorem}
\begin{proof}
	By the definition of the equivalence $\Psi$, as defined in the beginning of this subsection, we have the following commutative diagram with equivalences in the rows 
	\[ \xymatrix{ \Delta_{\CX}(R)\ar[rr]^{\Psi} && \K^{-,\rm{p}}(\CX) \\
		\K^{\rm{b}}_{\rm{ac}}(\CX)	\ar@{^(->}[u] \ar[rr]^{\Psi\mid} &&\K^{\rm{b}}_{\rm{ac}}(\CX) \ar@{^(->}[u]}\]
	In fact, the restriction of $\Psi$ on $\K^{\rm{b}}_{\rm{ac}}(\CX)$ in the above diagram is isomorphic to the identity functor. Now in view of the above diagram the functor $\Psi$ induces a triangle equivalence between the Verdier quotients $\frac{\Delta_{\CX}(R)}{\K^{\rm{b}}_{\rm{ac}}(\CX)}$ and $\frac{\K^{-,\rm{p}}(\CX)}{\K^{\rm{b}}_{\rm{ac}}(\CX)}$. But we know from Theorem \ref{Theorem 3.3}, the first Verdier quotient is indeed equivalent to $\D_{\rm{sg}}(R)$. So we are done.		
\end{proof}

Let us give another application of Proposition \ref{proposition 3.7}. Let $k$ be a  commutative artinian ring. We recall that a $k$-category $\mathcal{C}$ is said to be {\it Hom-finite} if for each two objects $X$ and $Y$ in $\mathcal{C}$, $\Hom_{\mathcal{C}}(X, Y)$ has finite length.\\
\begin{proposition} Assume $\La$ is an Artin $k$-algebra.
	Then the  category $ \frac{\K^{\rm{b}}(\CX)}{\K^{\rm{b}}(\rm{prj}\mbox{-}\La)}$ is $\rm{Hom}$-finite. In particular, If $\CX$ is a contravariantly finite subcategory in $\mmod \La$ and  $\La$ has finite global $\CX$-dimension, then the relative singularity category $\Delta_{\CX}(\La)$ is Hom-finite. 
\end{proposition}
\begin{proof}
Thanks to the equivalence given in Proposition \ref{proposition 3.7} we will show the statement for the triangulated category $\K^{-, \rm{p}}(\CX)$ instead.	As Remark \ref{remark1} says any complex in  $\K^{-, \rm{p}}(\CX)$ can be constructed via complexes of the form $\mathbf{C}_X$ for some $X \in \CX.$ Consequently, we only need to show that $\Hom_{\K}(\mathbf{C}_X[i], \mathbf{C}_Y[j])$ is of finite length (in $\mmod k$) for any $X, Y \in \CX$ and $i, j \in \Z.$ Without of loss generality, we may assume that $i=0.$ If $j>0$, then $\Hom_{\K}(\mathbf{C}_X, \mathbf{C}_Y[j])=0$, so nothing remains to prove. For $j=0$, it is clear to see that $\Hom_{\K}(\mathbf{C}_X, \mathbf{C}_Y)\simeq \underline{\rm{Hom}}_{\La}(X, Y),$ so we are done as the latter one is of finite length. 
It remains for  $ j< 0.$ Take $[f]$ in  $\Hom_{\K}(\mathbf{C}_X, \mathbf{C}_Y[j])$. Since $d^{j-1}_Y\circ f^0=0,$ hence  $f^0$ induces a morphism from $X$ to $\Omega_{\La}^{-j}(Y),$ say $g$. Mapping $[f]$ to $\underline{g}$ gives an isomorphism of $k$-modules $\Hom_{\K}(\mathbf{C}_X, \mathbf{C}_Y[j])\simeq \underline{\rm{Hom}}_{\La}(X, \Omega^{-j}_{\La}(Y)).$ The isomorphism completes the proof.
	\end{proof}

\section{ singular equivalences via relative singular and  stable equivalences  }
In this section by our results on relative singularity categories we will give some ways to determine two rings to be singular equivalence. We follow these two following observations. First, assume $\CX \subseteq \mmod R$ and $\CX' \subseteq \mmod R'$ be the same as in Theorem \ref{Theorem 3.10}. If there is a triangle equivalence between the relative singularity categories $\Delta_{\CX}(R)\simeq \Delta_{\CX'}(R')$, then it induces a triangle equivalence between the singularity categories $\D_{\rm{sg}}(R)\simeq\D_{\rm{sg}}(R')$, by Theorem \ref{Theorem 3.10}, in case that the equivalence functor can be restricted to the homotopy categories $\K^{\rm{b}}_{\rm{ac}}(\CX)$ and $\K^{\rm{b}}_{\rm{ac}}(\CX')$. Hence, we explore to find some conditions to have the restriction. This idea is exactly applied in \cite{KY}.
Another way, assume that  there is an equivalence between  the stable categories $\underline{\CX}$ and $\underline{\CX}'$, say $F.$ We know that by \cite[Theorem 3.1]{CC} we can  embed the stable categories $\underline{\CX}$ and $\underline{\CX'}$ in the quotient categories $\frac{\K^{\rm{b}}(\CX)}{\K^{\rm{b}}(\rm{prj}\mbox{-}R)}$ and $\frac{\K^{\rm{b}}(\CX')}{\K^{\rm{b}}(\rm{prj}\mbox{-}R')}$, respectively. The embeddings are defined in a canonical way, namely, by sending an object $X \in \CX$ to the complex concentrated at degree zero with term $X$, similarly for the case $\CX'$. What we will follow in this section is to see that how one can extend the equivalence $F$ under certain conditions to a triangle equivalence between the quotient categories, then by using  the first  observation to get a singular equivalence.\\
Let us begin with the following definition. 
\begin{definition}\label{Definition 4.1}
	We say that $\rm{prj}\mbox{-}R \subseteq \CX \subseteq \mmod R$ satisfies condition $(*)$: If for any $X \in \CX$, $\underline{\rm{Hom}}(Y, \Omega^n_R(X))=0$ for all $Y$ in $\CX$ and all but finitely many $n>0$, then the projective dimension of $X$ is finite. For the case that $R$ is an Artin algebra $\La$. Then, the condition $(*)$ is equivalent to say that the projective dimension of an indecomposable module $ X \in \CX$ is finite if and only if  any indecomposable module $Y$ in $\CX$ appears  up to isomorphism as a direct summand of finitely many of syzygies $\Omega^i_{\La}(X).$ Here we need to assume that the syzygies obtained by the minimal projective resolutions.
\end{definition} 
Following \cite[Definition 6.16]{KY}, we make the following definition.
\begin{definition}
	For a triangulated category $\CT$ the triangulated subcategory
	$$\CT_r:=\{X \in \CT \mid \ \Hom_{\CT}(Y, X[i])=0\ \  \text{ for all} \ Y \ \text{and all but finitely many} \ i \in \mathbb{Z}\}$$ is defined. More explanation, an  object $X$ is in $\CT_r$ if for any $Y$ in $\CT$, the Hom-spaces $\Hom_{\CT}(Y, X[i])$ is non-zero for only  finitely many $i$.	 
\end{definition}
The  $\CT_r$ in \cite{KY} is called the {\it subcategory of right homologically finite objects} of $\CT$.
\begin{proposition}\label{Prop 4.3}
Assume that $\rm{prj}\mbox{-}R \subseteq \CX \subseteq \mmod R$ satisfies the condition $(*)$ and quasi-resolving. Then,	$\K^{-, \rm{p}}(\CX)_r=\K^{\rm{b}}_{\rm{ac}}(\CX)$.
\end{proposition}
\begin{proof}
Let $\mathbf{X}$ be in $\K^{\rm{b}}_{\rm{ac}}(\CX)$ and $\mathbf{Y}$ in $\K^{-, \rm{p}}(\CX).$	Since $\mathbf{X}$ is a bounded complex then it is trivial that $\Hom_{\K}(\mathbf{Y}, \mathbf{X}[i])=0$ for sufficiently small $i.$ By the lifting property one can see that $\Hom_{\K}(\mathbf{Y}, \mathbf{X}[i])=0$ for $i>|m|+|n|$, where $m$, resp. $n,$  is the least degree such that $\mathbf{Y}^m$ is non-projective, resp. the biggest degree with  $\mathbf{X}^n\neq 0.$   Remember the complexes are assumed that  the differentials raise degree. So $\K^{\rm{b}}_{\rm{ac}}(\CX)\subseteq \K^{-, \rm{p}}(\CX)_r$. For the inverse inclusion, let $\mathbf{X} \in \K^{-, \rm{p}}(\CX)_r.$ Denote by $n$ the biggest degree with $\mathbf{X}^n\neq 0$ and $m$ the least degree such that $\mathbf{X}$ is non-projective.
For simplicity, we may assume that $n=0.$ For any $Y$ in $\CX$ and $i>|m|$, we infer $\Hom_{\K}(\mathbf{C}_Y[i], \mathbf{X})\simeq \underline{\rm{Hom}}_{R}(Y, \Omega^i_R(\ker d^m))$. Note that since $\CX$ is closed under kernels of epimorphisms then $\ker d^m$ belongs to $\CX.$ By the condition $(*)$  and in view of the isomorphisms we get the projective dimension of $\ker d^m$ is finite. This implies that $\mathbf{X}$ must be homotopy-equivalent to  a bounded complex, as desired.
\end{proof}
The above proposition  shows that $\K^{\rm{b}}_{\rm{ac}}(\CX)$ has a categorical characterization which is vital for our next result.\\
Let $\La$ be an Artin algebra. We say that a subcategory of $\mmod \La$ is of {\it finite representation type} if it has up to isomorphism only finitely many indecomposable modules. In case that $\mmod \La$ is of finite representation type, the $\La$ is called representation-finite.
\begin{proposition}\label{Proposition 4.3}
Assume $\La$ is an Artin algebra.	Let $\rm{prj}\mbox{-}\La\subseteq \CX \subseteq \mmod \La$ be closed under syzygies and of finite representation type. Then $\CX$ satisfies the condition $(*)$.
\end{proposition}
\begin{proof}
Let $X$ in $\CX$ satisfy the property in the condition $(*)$, i.e.,  any indecomposable module $Y$ in $\CX$ is isomorphic to a direct summand of finitely many syzygies $\Omega^i_{\La}(X)$.  Since $\CX$ is closed under syzygies then the set $\{\Omega^i_{\La}(X)\}_{i \geqslant 0}$	lies in $\CX.$ But since $\CX$ is of finite representation type, then there an indecomposable module $Y$ in $\CX$ such that it is isomorphic to a direct summand of  $\Omega^i_{\La}(X)$ for infinitely many $i.$ This give  a contradiction. So we are done.
\end{proof}
We remind by our convention any subcategory in the paper is assumed to be closed under direct summands. We use this convention in the proof of the above proposition. 
\begin{proposition}\label{Proposition 4.4}
	Assume that $ \rm{prj}\mbox{-}R \subseteq \CX \subseteq \mmod R$ and $\rm{prj}\mbox{-}R' \subseteq\CX' \subseteq \mmod R' $ are contravariantly finite subcategories in the corresponding ambient categories. Further, $R$, resp. $R'$, has finite global $\CX$-dimension, resp. $\CX'$-dimension, and both subcategories satisfy the condition $(*)$ and being quasi-resolving. Then, if there is a triangle  equivalence $\Delta_{\CX}(R)\simeq \Delta_{\CX'}(R')$ of triangulated categories, then there is a triangle equivalence $\D_{\rm{sg}}(R)\simeq \D_{\rm{sg}}(R').$ 
\end{proposition}
\begin{proof}
	The statement is a consequence of Proposition \ref{Prop 4.3} and Theorem \ref{Theorem 3.10}. In fact, a triangle equivalence $\Delta_{\CX}(R)\simeq \Delta_{\CX'}(R')$, because of the intrinsic characterization given in Proposition \ref{Prop 4.3}, restrict to a triangle equivalence between $\K^{\rm{b}}_{\rm{ac}}(\CX)\simeq \K^{\rm{b}}_{\rm{ac}}(\CX')$. Thus it induces a triangle equivalence between the associated Verdier quotient categories $\Delta_{\CX}(R)/\K^{\rm{b}}_{\rm{ac}}(\CX)$ and $\Delta_{\CX'}(R')/\K^{\rm{b}}_{\rm{ac}}(\CX')$. By Theorem \ref{Theorem 3.10}, we observe that the Verdier quotient categories are equivalent to the respective singularity categories, so the result.
\end{proof}
In the next we give some examples satisfying the assumptions of the above proposition.
\begin{example}\label{Example 5.5}
	\begin{itemize}
		\item [$(1)$] Let $\La$ be a $\rm{CM}$-finite Gorenstein algebra. Recall that $\La$ is called $\rm{CM}$-finite, if $\rm{Gprj}\mbox{-}\La$ has only
		finitely many isomorphism classes of indecomposable modules. It is known that $\rm{Gprj}\mbox{-}\La$ is closed under kernels of epimorphisms, see e.g. \cite[Proposition 2.1.7]{C4}. Then, in view of  Proposition \ref{Proposition 4.4}, $\rm{Gprj}\mbox{-}\La$ satisfies the all conditions we need in Proposition \ref{Proposition 4.4}.
		\item[$(2)$] Let $\La$ be an Artin algebra with radical square zero, that is $\rm{rad}^2(\La)=0$. Then the additive closure $S(\La)=\rm{add}\mbox{-}(\La\oplus \frac{\La}{\rm{rad}(\La)})$  of projective modules and simple modules  verifies the conditions. To check, only use this fact that  submodules of projective modules are semisimple over Artin algebras with radical square zero. 
		\item[$(3)$] Assume $\La$ is representation-finite. Then, it is obvious that  $\mmod \La$ satisfies the assumptions.	
	\end{itemize}
Specializing Proposition \ref{Proposition 4.4} for the above examples follows:
\begin{corollary}
Let $\La$ and $\La'$ be two Artin algebras. The following statement happen.
\begin{itemize}	
	\item [$(1)$]Assume that $\La$ and $\La'$ are $\rm{CM}$-finite Gorenstein. Then, if there is a triangulated equivalence $\Delta( \rm{Gprj}\mbox{-}\La)\simeq \Delta(\rm{Gprj}\mbox{-}\La')$, then $\D_{\rm{sg}}(\La)\simeq \D_{\rm{sg}}(\La').$
	\item[$(2)$] Assume that $\La$ and $\La$ are algebras with radical squarer zero. Then, if there is a triangulated equivalence $\Delta(S(\La))\simeq \Delta(S(\La'))$, then $\D_{\rm{sg}}(\La)\simeq \D_{\rm{sg}}(\La').$
	\item [$(3)$] Assume that $\La$ and $\La$ are representation-finite. Then, if there is a triangulated equivalence $\Delta(\mmod \La)\simeq \Delta(\mmod \La')$, then $\D_{\rm{sg}}(\La)\simeq \D_{\rm{sg}}(\La').$
	\end{itemize}	
	\end{corollary}

\end{example}
 In concern of the second observation given in the beginning of this section we establish the following construction.
 \begin{construction}\label{Costruction 2}
 	Assume that $ \rm{prj}\mbox{-}R \subseteq \CX \subseteq \mmod R$ and $\rm{prj}\mbox{-}R' \subseteq\CX' \subseteq \mmod R' $ are contravariantly finite subcategories in the corresponding ambient categories and closed under syzygies. Moreover, there is an equivalence $\underline{F}:\underline{\CX}\rt \underline{\CX}'$ which is induced from a functor $F:\CX\rt \CX'$ preserving projective modules and syzygies, i.e., for any $X \in \CX$, $F(\Omega_R(X))\simeq \Omega_{R'}(F(X))$ in $\underline{\CX}'.$
Since  the functor $F$ is additive then it induces the triangle functor $\K^{\rm{b}} F$ between the bounded  homotopy categories $\K^{\rm{b}}(\CX)$ and $\K^{\rm{b}}(\CX')$. On the other hand,  $F$ preserves projective modules then the  functor $\K^{\rm{b}}F$ preserves the bounded  homotopy category of projective modules as well. Hence $\K^{\rm{b}}F$ induces a triangle functor $ \widetilde{\K^{\rm{b}}F}: \frac{\K^{\rm{b}}(\CX)}{\K^{\rm{b}}(\rm{prj}\mbox{-}R)} \rt  \frac{\K^{\rm{b}}(\CX')}{\K^{\rm{b}}(\rm{prj}\mbox{-}R')}$. Denote by $\K^{\rm{p}}F:=\Psi_{\CX'}\circ \widetilde{\K^{\rm{b}}F} \circ \Psi^{-1}_{\CX}$, see the beginning of Section 3 for the definitions of $\Psi_{\CX}$ and $\Psi_{\CX'}$. 
The  introduced  functors satisfy the following commutative diagram  
 \begin{equation*}
 \begin{array}{c}
 \begin{xy}\SelectTips{cm}{}
 \xymatrix{ 	
 	\underline{\CX} \ar@{^{(}->}[r]\ar[d]^{\underline{F}} & \frac{\displaystyle \K^{\rm{b}}(\CX)}{\displaystyle \K^{\rm{b}}(\rm{prj}\mbox{-}R)} \ar[r]^{\Psi_{\CX}}\ar[d]^{\widetilde{\K^{\rm{b}}F}}  &  \K^{-, \rm{p}}(\CX)\ar[d]^{\K^{\rm{p}}F} \\
 	\underline{\CX}'\ar@{^{(}->}[r]	&\frac{\displaystyle \K^{\rm{b}}(\CX')}{\displaystyle \K^{\rm{b}}(\rm{prj}\mbox{-}R')}  \ar[r]^{\Psi_{\CX'}} &
 	\K^{-, \rm{p}}(\CX').
 }\end{xy}
 \end{array}
 \end{equation*}
 Assume, further, $\CX$ and $\CX'$ are quasi-resolving, and the functor $F$ is exact, meaning that any short exact sequence with terms in $\CX$ is mapped by $F$ into a short exact sequence in $\CX'.$ By applying terms by terms the functor $F$ and using its exactness,  we can extend the functor $F$ to $\K^{\rm{p}}_{\rm{e}}F:\K^{-, \rm{p}}(\CX)\rt \K^{-, \rm{p}}(\CX')$. By the construction, we clearly have the following commutative diagram 
 
  \begin{equation*}
 \begin{array}{c}
 \begin{xy}\SelectTips{cm}{}
 \xymatrix{ 	
 	\K^{\rm{b}}_{\rm{ac}}(\CX) \ar@{^{(}->}[r]\ar[d]^{\K^{\rm{p}}_{\rm{e}}F\mid} &\K^{-, \rm{p}}(\CX) \ar[d]^{\K^{\rm{p}}_{\rm{e}}F}   \\
 	\K^{\rm{b}}_{\rm{ac}}(\CX)\ar@{^{(}->}[r]	&\K^{-, \rm{p}}(\CX') .}\end{xy}
 \end{array}
 \end{equation*}
 Thus because  of the above diagram $\K^{\rm{p}}_{\rm{e}}F$ induces a triangle functor between the corresponding Verdier  quotient categories  $\frac{\displaystyle \mathbb{K}^{-, \rm{p}}(\CX)}{ \displaystyle \mathbb{K}^{\rm{b}}_{\rm{ac}}(\CX)}$ and $ \frac{\displaystyle \mathbb{K}^{-, \rm{p}}(\CX')}{ \displaystyle \mathbb{K}^{\rm{b}}_{\rm{ac}}(\CX')}$. Now by considering Theorem \ref{Theorem 3.10}, we obtain a triangle functor from $\D_{\rm{sg}}(R)$ to $\D_{\rm{sg}}(R')$, assuming the global $\CX$-dimension, resp. $\CX'$-dimension,  of $R$, resp. $R',$ is finite. To have the functor $\K^{\rm{p}}_{\rm{e}}F$, we can drop the assumption to be closed under kernels of epimorphisms (being quasi-resolving) for the involved subcategories; set another assumption on $F$. That is $F$ to be a restriction of an exact functor from a quasi-resolving subcategory $\CY \subseteq \mmod R$ containing $\CX$ to a quasi-resolving subcategory $\CY' \subseteq \mmod R'$ containing $\CX'.$  As a special case, when $F$ is a restriction of an exact functor from $\mmod R$ to $\mmod R'$. In this case no need to assume that  the functor $F$ to preserve the syzygies, since the exactness of the functor automatically implies it.
  \end{construction}
 \begin{proposition}\label{Proposition 4.9}
 	Keep in mind all the notations in  the above construction. Then the following statements hold.
 	\begin{itemize}
 	\item [$(1)$] The functor $\K^{\rm{p}}F$ is a triangle equivalence.
 	\item [$(2)$] Assume $\CX$ and $\CX'$ are quasi-resolving. Then, $\K^{\rm{p}}_{\rm{e}}F$ is a triangle equivalence. 
 	 \item[$(3)$] Assume $F$ is a restriction of an exact functor between  quasi-resolving subcategories  $\CX \subseteq \CY \subseteq \mmod R$ and $\CX' \subseteq \CY' \subseteq \mmod R'$. Then, $\K^{\rm{p}}_{\rm{e}}F$ is a triangle equivalence. In this situation we still use the notation $\K^{\rm{p}}_{\rm{e}}F$, it is indeed the restricted functor $\K^{\rm{p}}_{\rm{e}}G\mid: \K^{-, \rm{p}}(\CX)\rt \K^{-, \rm{p}}(\CX')$, where $G$ is the exact functor between $\CY$ and $\CY'$.
 	 \end{itemize}
 \end{proposition}
\begin{proof}
 Since the subcategory $\mathcal{C}$ consisting   of all $\mathbf{C}_X[i]$ for all $X \in \CX$ and $i \in \Z$ generates $\K^{-, \rm{p}}(\CX)$, see Remark \ref{remark1}, then, by using a standard argument, it is enough to show that the restriction of  $\K^{\rm{p}}F$ on $\mathcal{C}$ is an equivalence. Note that by the definition the image of $\mathcal{C}$ under $\K^{\rm{p}}F$ is mapped into the subcategory  $\mathcal{C}'$ in $\K^{-, \rm{p}}(\CX')$ which is defined similar to $\mathcal{C}.$ Let $\mathbf{C}_{X'}[i]$, for some $i \in \Z$ and $X \in \CX',$ be in $\mathcal{C}'$. Due to the equivalence $\underline{F}$ there is a module $X$ in $\CX$ with $\underline{F}(X)\simeq X'$. Then by the uniqueness of projective resolution up to homotopy equivalence we can see that $\K^{\rm{p}}F(\mathbf{C}_X[i])= \mathbf{C}_{\underline{F}(X)}[i]\simeq \mathbf{C}_{X'}[i]$. So this proves the denseness of the restricted functor. For any pair of modules $A$ and $B$ in $\CX$ and $i \in \Z,$  we have the following natural  isomorphisms
	\begin{itemize} 
		\item [$(1)$]If $i=0,$  \begin{align*}	
		 \Hom_{\K(R)}(\mathbf{C}_A, \mathbf{C}_B)&\simeq  \underline{\rm{Hom}}_{R}(A, B)\\ &\simeq \underline{\rm{Hom}}_{R'}(F(A), F(B))\\ &\simeq \Hom_{\K(R')}(\mathbf{C}_{F(A)}, \mathbf{C}_{F(B)}) \\ &= \Hom_{\K(R')}(\K^{\rm{p}}F(\mathbf{C}_A), \K^{\rm{p}}F(\mathbf{C}_B))  \end{align*}
		\item [$(2)$] If $i>0$, then \begin{align*}	
		 \Hom_{\K(R)}(\mathbf{C}_A[i], \mathbf{C}_B)&\simeq \underline{\rm{Hom}}_{R}(A, \Omega^i_R(B)) \\ &\simeq \underline{\rm{Hom}}_{R'}(F(A), F(\Omega^i_R(B))) \\ &\simeq
		 \underline{\rm{Hom}}_{R'}(F(A), \Omega^i_{R'}(F(B))) \\& \simeq 
		  \Hom_{\K(R')}(\mathbf{C}_{F(A)}[i], \mathbf{C}_{F(B)})\\ &= \Hom_{\K(R')}(\K^{\rm{p}}F(\mathbf{C}_A[i]), \K^{\rm{p}}F(\mathbf{C}_B)) \end{align*}

		\item [$(2)$] If $i <0$, then $$\Hom_{\K(R)}(\mathbf{C}_A[i], \mathbf{C}_B)\simeq 0 \simeq  \Hom_{\K(R')}(\mathbf{C}_{F(A)}[i], \mathbf{C}_{F(B)})=\Hom_{\K(R')}(\K^{\rm{p}}F(\mathbf{C}_A[i]), \K^{\rm{p}}F(\mathbf{C}_B)).$$	
	\end{itemize}
The above isomorphisms show that the restriction of $\K^{\rm{p}}F$ on $\mathcal{C}$ is full and faithful. So we are done. The statements $(2)$ and $(3)$ can be proved by the same argument. So we skip their proofs.
\end{proof}
\begin{theorem}\label{Coroallry 5.7}
	Let $\CX \subseteq \mmod R$ and $\CX' \subseteq \mmod R'$ be as in Construction \ref{Costruction 2}, that is,  $ \rm{prj}\mbox{-}R \subseteq \CX \subseteq \mmod R$ and $\rm{prj}\mbox{-}R' \subseteq \CX' \subseteq \mmod R' $ are contravariantly finite  and closed under syzygies. Assume, further, the global $\CX$-dimension, resp. $\CX'$-dimension,  of $R$, resp. $R'$, is finite. Suppose there is a functor $F:\CX\rt \CX'$ such that $F(\rm{prj}\mbox{-}R) \subseteq \rm{prj}\mbox{-}R'$, the induced functor $\underline{F}:\underline{\CX}\rt \underline{\CX}'$ is an equivalence, and  for any $X \in \CX$, $F(\Omega_R(X))\simeq \Omega_{R'}(F(X))$ in $\underline{\CX}'$ (as in Construction \ref{Costruction 2}). If either of the following situations happens.
	\begin{itemize}
\item [$(1)$]	
The subcategories  $\CX$ and $\CX'$ satisfy the condition $(*)$. 
	 \item [$(2)$] The subcategories  $\CX$ and $\CX'$ are quasi-resolving, and the functor $F$ is exact. 
	 \item [$(3)$] The functor $F$ is a restriction of an exact functor between  quasi-resolving subcategories  $\CX \subseteq \CY \subseteq \mmod R$ and $\CX' \subseteq \CY' \subseteq \mmod R'$.
	 \end{itemize} 
 Then, $R$ and $R'$ are singularly equivalent.
\end{theorem}
\begin{proof}
	For $(1)$, we know by Proposition \ref{Proposition 4.9} the functor $\K^{\rm{p}}F$, defined in the construction, is an equivalence between the relative singularity categories $\Delta_{\CX}(R)\simeq \Delta_{\CX'}(R')$. Thanks to Proposition \ref{Proposition 4.4} we get the result since  the subcategories  satisfy the condition $(*)$. The proofs of  $(2)$ and $(3)$ come from this fact that the triangle equivalence  $\K^{\rm{p}}_{\rm{e}}F$ is restricted to $\K^{\rm{b}}_{\rm{ac}}(\CX)\simeq \K^{\rm{b}}_{\rm{ac}}(\CX')$, by the construction. Hence $\K^{\rm{p}}_{\rm{e}}F$ induces similarly an equivalence between the singularity categories. So we are done these two cases. 
\end{proof}
As the above theorem says that some kinds of relative stable equivalences implies the singular equivalences. An interesting cases of $\CX$ and $\CX'$ is when both equal to the whole of the module categories. In this special case we are dealing with   the stable equivalences which have been worked by many researchers. 

 \begin{corollary}\label{Corolalry 5.10}
 	Let $R$ and $R'$ be two right notherian rings. Let $F:\mmod R \rt \mmod R'$ be a functor with $F(\rm{prj}\mbox{-}R) \subseteq \rm{prj}\mbox{-}R'.$ Suppose that $F$
 induces the stable equivalence $\underline{F}:\underline{\rm{mod}}\mbox{-}R\rt \underline{\rm{mod}}\mbox{-}R'$ with  the property $\underline{F}(\Omega_R(M))\simeq \Omega_{R'}(\underline{F}(M))$ for any $R$-module $M$ in $\mmod R$. If one of the two following assumptions is correct.
 \begin{itemize}
 	\item [$(1)$] The rings $R$ and $R'$ are two Artin algebras of finite representation type.
 	\item [$(2)$] The functor $F$ is exact
 \end{itemize}
   Then, $R$ and $R'$ are singularly equivalent .
 \end{corollary}
\begin{proof}
	It is clear that the relative dimension respect to the whole of module category is always finite. Then, the statement is an immediate consequence of Theorem \ref{Coroallry 5.7}.
	\end{proof}

In the next section we will provide different examples of the subcategories satisfying the assumptions of our  main theorem (Theorem \ref{Coroallry 5.7}). Let us here only focus on the case that the subcategories  to be the whole of the  module categories (as the above corollary) and conclude this section by some remarks concerning this special case.

\begin{remark}\label{lastremark}
\begin{itemize}
\item [$(1)$]
In the case $(2)$ of Corollary \ref{Corolalry 5.10}, we are dealing with the  stable equivalences  induced by exact functors.  We refer to \cite{L} for more information related to  such  stable equivalences. Note that in the  contrast of \cite{L} we do not need the quasi-inverse of  the functor $\underline{F}$ (appeared in  Corollary \ref{Corolalry 5.10})  also induced by an exact functor. An important example of the stable equivalences induced by exact functors is stable equivalences of Morita type. There is  a series of works due to Yuming Liu and Changchang Xi \cite{LX1, LX2, LX3} to give some ways to construct this sort of the stable equivalences and which of  properties are preserved. Recall two algebras $\La$ and $\La'$ are said to be {\it stable equivalence of Morita type} if there exist a pair of bimodules ${}_{\La}M_{\La'}$ and ${}_{\La'}N_{\La}$ such that
\begin{itemize}
	\item [$(1)$] $M$ and $N$ are projective as left and right modules, respectively;
	\item[$(2)$]  $M\otimes_{\La'}N\simeq A\oplus P$   as $A\mbox{-}A$-bimodules for some projective $A\mbox{-}A$-bimodule $P$, 	and $N \otimes_{\La}M\simeq B\oplus Q $  as $B\mbox{-}B$-bimodules for some projective $B\mbox{-}B$-bimodule $Q$.
\end{itemize}
By the definition one can prove directly a stable equivalence of Morita type gives also a singular equivalence. The notion of stable equivalences of Morita type is generalized to {\it singular equivalence of Morita type} in \cite{CS} and also worked in \cite{ZZ}.  Corollary \ref{Corolalry 5.10}(2) explains how one can get a singular equivalence from a certain given stable equivalence, as happens for the stable equivalences of Morita type. Moreover, the above corollary also can be obtained by using the stabilization theory established by Keller-Vossieck, Beligiannis \cite{B, KV}. We also refer the reader to \cite{C6}  for a nice survey of this theory. For an exact  functor $F:\mmod R\rt \mmod R'$ as in Corollary \ref{Corolalry 5.10}(2), the induced functor $\underline{F}$ is indeed a (left) triangle functor between the  left triangulated categories $\underline{\rm{mod}}\mbox{-}R$ and $\underline{\rm{mod}}\mbox{-}R'$. But, by \cite[Corolary 2.7]{C2} there is a triangle equivalence between the  relevant stabilization $\CS(\underline{\rm{mod}}\mbox{-}R)\simeq \CS(\underline{\rm{mod}}\mbox{-}R').$ We know from the basic result due to \cite{KV} the triangle equivalences $\CS(\underline{\rm{mod}}\mbox{-}R)\simeq \D_{\rm{sg}}(R)$ and $\CS(\underline{\rm{mod}}\mbox{-}R')\simeq \D_{\rm{sg}}(R')$, so we get  the corollary (with assumption of exactnesses of the functor $F$). This observation was mentioned to us by Xiao-Wu Chen. We would like to thank him. In case that the functor $F$ is not exact it seems that the stabilization theory does not work.
\item [$(2)$] It is not true any stable equivalence between two rings gives an singular equivalence between them. For instance, it is known any algebra with radical square zero is stable  equivalent to a hereditary algebra (\cite[Theorem 2.1]{AR}). But the singularity categories of hereditary algebras is always trivial, but not to be trivial in general for radical square zero algebras.
\item [$(3)$] Let $\La=k[x]/(x^2)$ be the algebra of dual numbers. The assignment $M \mapsto (M, \rm{soc}(M))$ gives a functor $F: \mmod \La \rt \mmod \La \times \mmod k.$ This functor satisfies the required assumption in Corollary \ref{Corolalry 5.10}(1), but not necessarily $F$ to be exact. 
\end{itemize}

\end{remark}

\section{ consequences and examples}
We aim in this section to give some application of our results. In different types of rings, including path rings, triangular matrix rings, trivial extension rings and tensor rings, by use of Theorem \ref{Theorem 3.3}, a new description of their singularity categories is given. Then by help of such description we construct some singular equivalences.
\subsection{Path rings }\label{Subsection 3.1}
Throughout let $\CQ$ be a finite acyclic  quiver $\CQ=(\CQ_0, \CQ_1, s, t)$, where $\CQ_0$ and $\CQ_1$ are the set of arrows and vertices of $\CQ$, respectively, and $s$ and $t$ are the starting and ending maps from $\CQ_1$ to $\CQ_0,$ respectively. Assume that a  right notherian ring $R$ is given, a representation $X$ of $\CQ$ over  $ R$ is obtained by associating to any vertex $v$ a module $X_v$ in $\mmod R$ and to any arrow $a:v \rt w$ a morphism $X_a: X_v \rt X_w$ in $\mmod R.$ If $\CX$ and $\CY$ are two representations of $\CQ$, then  a morphism $f:X \rt Y$ is determined by a family $\{f_v\}_{v \in \CQ_0}$  so that for any arrow $a:v \rt w$, the commutativity condition $Y_a \circ f_v=f_w \circ X_a$ holds. The representations of $\CQ$ over $\mmod R$ and the  morphisms between them, as defined already,  form an abelian  category which is denoted by $\rm{rep}(\CQ, R).$ We can also construct  the path ring of $\CQ$ by $R$ as follows. Let $\rho$ be the set of all path in the given quiver $\CQ$ together with  the trivial paths associated to the vertices. We write the conjunction of paths from left to right. Now let $R \CQ$ be the free $R$-module with basics $\rho.$ An element of $ R \CQ$ is written as a finite sum $\sum_{p \in \rho}a_P p$, where $a_p \in R$ and $a_{\rho}=0$ for all but finitely many $\rho.$ We can make $R\CQ$ a ring where  multiplication is given by  concatenation of paths. Then $R \CQ$ is still a right notherian ring, and as it is a free module of finite rank over the right notherian ring $R$.  One can prove in the same way for the case $R$ to be  a filed $k$ as in the literature, see for example \cite{ARS}, the category $\mmod R\CQ$ of (right) finitely generated  $R\CQ$-modules is equivalent to the category $\rm{rep}(\CQ, R)$ of representations of $\CQ$ by $R$-modules and $R$-homormorphisms. Hence due to this equivalence we identify $\mmod R\CQ$ with  $\rm{rep}(\CQ, R)$ which is much easier to work. \\
For any vertex $v$ in $V$ of quiver $\CQ$, let $e^v:\rm{rep}(\CQ, R)\rt \mmod R$ be the evaluation functor defined by $e^v(X)=X_v,$ for a representation $X$. It is shown  in \cite{EH} that $e^v$ has a full faithful left adjoint $e^v_{\lambda
}:\mmod R \rt \rm{rep}(\CQ, R)$, by sending $M$ in $\mmod R$ to the representation $e^v_{\la}(M)$, that is, for a vertex $w \in V$, $e^v_{\lambda}(M)(w)=\oplus_{Q(v, w)}M$, where $Q(v, w)$ is the set of all paths from $v$ to $w,$ for an arrow $a:w_1\rt w_2$, $e^v_{\la}(M)(a)$ is the natural injection from
$\oplus_{Q(v, w_1)}M$ to $\oplus_{Q(v, w_2)}M$. As mentioned in \cite[subsection 3.1]{EHS} there is a short exact sequence 
$$\dagger \ \ \ \ \  \ \ \ 0 \rt \bigoplus_{a \in E} e^{t(a)}_{\lambda}(X_{s(a)})\st{g_X}\rt  \bigoplus_{v \in V} e^{v}_{\lambda}(X_v) \st{f_X} \rt X \rt 0.$$
The above short exact sequence is vital for our next observation.\\
Let $\CM_{R}(\CQ)$ be the subcategory of $\rm{rep}(\CQ, R)$ consisting of all representations being isomorphic to a  finite direct sum of  representations in the form of $e^v_{\lambda}(M)$ for some $v$ in $V$ and $M$ in $\mmod R.$ By the classification of projective representations given in \cite[Theorem 3.1]{EE} we observe $\CM_{R}(\CQ)$ contains projective representations. For instance, in the case that $\CQ$ is the quiver  $A_2:1\rt 2$, then the $\CM_R(A_2)$ is formed of all representations which are isomorphic to a finite direct sum of representations in the form of either $0\rt M$ or $M\st{1}\rt M$, for some $M $ in $\mmod R.$ Equivalently, it contains all representations $M_1\st{f}\rt M_2$ such that $f$ is a split monomorphism.\\
We usually use $\Hom_{\CQ}(-, -)$ to show the Hom-space between two representations. 
\begin{lemma}\label{lemma 6.1}
	Use the above notations. The subcategory $\CM_{R}(\CQ)$ is contravariantly finite in $\rm{rep}(\CQ, R)$. Moreover, the global $\CM_{R}(\CQ)$-dimension of $R\CQ$ is  at most one.
\end{lemma}
\begin{proof}
	We first consider the first statement. To this end, we show that the morphism $f_X:\bigoplus_{v \in V}e^v_{\lambda}(X)\rt X$ appeared in the short exact sequence $(\dagger)$, introduced in the above, is a right $\CM_{R}(\CQ)$-approximation. Let us first present $f_X$ more explicitly as  $(f^v_X)_{v \in V}$, where $f^v_X:e^v_{\la}(X_v)\rt X$ is the image of the identity $\rm{id}_{X_v}$ under the isomorphism $\Hom_{\CQ}(e^v_{\la}(X_v), X)\simeq \Hom_{R}(X_v, X_v)$ for any $v$ in $V.$ To prove, it is enough to show that every morphism as $h: e^w_{\la}(M)\rt X$, $M \in \mmod R$ and $w \in V$, factors through $f_X.$ Consider the following natural isomorphisms
	$$\Hom_{\CQ}(e^w_{\la}(M), X)\simeq \Hom_{R}(M, X_w)\simeq \Hom_{\CQ}(e^w_{\la}(M), e^w_{\la}(X_w)).$$
	Denote $g:e^w_{\la}(M)\rt e^w_{\la}(X_w)$ the image of $h$ under the composition of the above two isomorphisms, that is indeed, $e^w_{\la}(d)$, where $d:M\rt X_w$ is the image of $h$ under the first isomorphism. Then by using the adjoint property we observer $f^w_X\circ g=h,$ so $h$ factors through the $f_X$. The latter part immediately follows from the short exact sequence $(\dagger).$
\end{proof}
Fortunately, by the above lemma, $\CM_{R}(\CQ)$ satisfies all the conditions we need to use Theorem \ref{Theorem 3.10}. Hence as an immediate consequence of the theorem we have the following result.
\begin{theorem}\label{Theorem 6.3}
	Let $\CQ$ be an acyclic quiver and $R$ a right notherian ring. Then, there is the following triangle equivalence
	$$\D_{\rm{sg}}(R\CQ)\simeq \frac{\K^{-, \rm{p}}(\CM_{R}(\CQ))}{\K^{\rm{b}}_{\rm{ac}}(\CM_{R}(\CQ))}.$$
\end{theorem}
One advantage of the above equivalence is that $\D_{\rm{sg}}(R\CQ)$ is completely determined by the representations in the form of $e^v_{\la}(M)$ which have a  more simple structure and easy to handle some difficulties as we shall show in our next result. 
Let $F:\mmod R \rt \mmod R' $ be an exact functor the same as in Corollary \ref{Corolalry 5.10}(2), that is, $F(\rm{prj}\mbox{-}R)\subseteq \rm{prj}\mbox{-}R'$ and the stable induced functor $\underline{F}:\underline{\rm{mod}}\mbox{-}R\rt \underline{\rm{mod}}\mbox{-}R'$ is an equivalence. By applying component-wisely $F$, we get an exact  functor $\widetilde{F}:\rm{rep}(\CQ, R)\rt \rm{rep}(\CQ, R')$, preserving projective representations.
The preservation comes from this fact that for any $M \in \mmod R$ and $v \in V$, by the construction, we have $\tilde{F}(e^v_{\la}(M))=e^v_{\la}(F(M))$. Let us explain our construction for the case $\CQ$ is $A_2: 1\rt 2$.  For each representation $A\st{f}\rt B$ in $\rm{rep}(\CQ, R)$, $\tilde{F}(A\st{f}\rt B)=F(A)\st{F(f)}\rt F(B)$, and on morphisms can be defined similarly. The natural question may arise here is that whether $\tilde{F}$ induces an equivalence between the stable categories $\underline{\rm{rep}}(\CQ, R)$ and $\underline{\rm{rep}}(\CQ, R')$, or the same as $\underline{\rm{mod}}\mbox{-}R\CQ$ and $\underline{\rm{mod}}\mbox{-}R'\CQ$. Many difficulties it seems  one may occur, and of course we do not know  whether it is true or not at present. For example, if the functor $F$ is the tensor functor induced by a stable equivalence of Morita type, then one can show directly the $\tilde{F}$ makes an equivalence for the stable categories of the path rings as well. But in general it is not clear now. Surprisingly, the induces  quiver functor $\tilde{F}$  enables us to prove  the existence of  a triangle equivalence between $\D_{\rm{sg}}(R\CQ)$ and $\D_{\rm{sg}}(R'\CQ)$.  Although, the  structure of the singularity category $\D_{\rm{sg}}(R\CQ)$ is more complicated than the stable category $\underline{\rm{mod}}\mbox{-}R\CQ.$  This is mostly because of the  easy description of $\D_{\rm{sg}}(R\CQ)$ given in Theorem \ref{Theorem 6.3}. Based on Construction \ref{Costruction 2} since $\tilde{F}$ is an exact functor then there exists  the functor $\K^{\rm{p}}_{\rm{e}}\tilde{F}:\K^{-, \rm{p}}(\CM_{R}(\CQ))\rt \K^{-, \rm{p}}(\CM_{R'}(\CQ))$ obtained by applying $\tilde{F}$ terms by terms. In the rest,  we show that $\K^{\rm{p}}_{\rm{e}}\tilde{F}$ is a triangle equivalence.
For simplicity, denote $\K(\CQ(R))$ for the homotopy category of $\rm{rep}(\CQ, R).$ The functors $e^v$ and $e^v_{\la}$ naturally, by applying terms by terms, can be extended to the corresponding homotopy categories $\K(\mmod R)$ and $\K(\CQ(R))$, denoted, respectively, by $k^v$ and $k^v_{\la}$. The extended triangle functors form  an adjoint pair $(k^v_{\la}, k^v)$ as well. \\
Consider the following easy lemma.
\begin{lemma}\label{Lemma 6.3}
	Let $M$ be in $\mmod R$ and $v$ in $V.$  Then, we have the equality $\K^{\rm{p}}_{\rm{e}}\tilde{F}(k^v_{\la}(\mathbf{C}_N))= k^v_{\la}(\K^{\rm{p}}_{\rm{e}}F(\mathbf{C}_N))$. Recall that  $\mathbf{C}_N$ is the mapping cone  of an onto  chain map from a projective resolution $\mathbf{P}_N$ to $N$, see the beginning of Section 3, and also $\K^{\rm{p}}_{\rm{e}}F$ is defined in Construction \ref{Costruction 2}.
\end{lemma}
Now we prove the promised singular equivalence.
\begin{theorem}
	Use the above notation. Then, the induced functor  $\K^{\rm{p}}_{\rm{e}}\tilde{F}:\K^{-, \rm{p}}(\CM_{R}(\CQ))
	\rt \K^{-, \rm{p}}(\CM_{R'}(\CQ))$ is a triangle equivalence. In particular, $\D_{\rm{sg}}(R\CQ)\simeq \D_{\rm{sg}}(R'\CQ)$ as triangulated categories.
\end{theorem}
\begin{proof}
	As we did before in the proof of Proposition \ref{Proposition 4.9}, it is enough to show that the restricted functor $\K^{\rm{p}}_{\rm{e}}\tilde{F}:\CM_R(\CQ)_0\rt \CM_{R'}(\CQ)_0$ is an equivalence. Recall that $\CM_R(\CQ)_0$ is the subcategory of  $\K^{-, \rm{p}}(\CM_{R}(\CQ))$ consisting of all shiftings of the complexes in the form of $\mathbf{C}_{e^v_{\la}(M)}$ for some $M \in \mmod R$ and $v \in V,$ and similarly for $\CM_{R'}(\CQ)_0$. We recall again that  $\mathbf{C}_{e^v_{\la}(M)}$  is an exact complex obtained from  a projective resolution of $e^v_{\la}(M)$.  We may assume that  $\mathbf{C}_{e^v_{\la}(M)}$ is  obtained by applying $k^v_{\la}$  on $\mathbf{C}_M$. This makes a sense because $e^v_{\la}$ is an exact functor preserving projective representations.
	 For density, take $\mathbf{C}_{e^v_{\la}(N)}[i]$ in $\CM_{R'}(\CQ)_0$, where $N \in \mmod R'$. By the equivalence $\underline{F}$, there is $M$ in $\mmod R$ such that $\underline{F}(M)\simeq N.$ Then one can see by Lemma \ref{Lemma 6.3}, $\K^{\rm{p}}_{\rm{e}}\tilde{F}(\mathbf{C}_{e^v_{\la}(M)}[i])\simeq \mathbf{C}_{e^v_{\la}(N)}[i].$ Let $\mathbf{C}_{e^w_{\la}(T)}[j]$ and $\mathbf{C}_{e^v_{\la}(M)}[i]$ be arbitrary  complexes in $\CM_R(\CQ)_0$. There is the following chain of natural isomorphisms
	 	\begin{align*}
	 \Hom_{\K(\CQ(R))}(\mathbf{C}_{e^v_{\la}(M)}[i], \mathbf{C}_{e^w_{\la}(T)}[j] )&= \Hom_{\K(\CQ(R))}(k^v_{\la}(\mathbf{C}_M)[i], k^w_{\la}(\mathbf{C}_T)[j]) \\
	 &\simeq \Hom_{\K(R)}(\mathbf{C}_M[i], \oplus_{\CQ(v, w)}\mathbf{C}_T[j]) \\ &\simeq \Hom_{\K(R')}(\K^{\rm{p}}_{\rm{e}}F(\mathbf{C}_M[i]), \oplus_{\CQ(v, w)}\K^{\rm{p}}_{\rm{e}}F(\mathbf{C}_T[j]))\\
	 &\simeq \Hom_{\K(\CQ(R'))}(k^v_{\la}(\K^{\rm{p}}_{\rm{e}}F(\mathbf{C}_M[i])), k^w_{\la}(\K^{\rm{p}}_{\rm{e}}F(\mathbf{C}_T[j])))\\&\simeq \Hom_{\K(\CQ(R'))}(\K^{\rm{p}}_{\rm{e}}\tilde{F}(\mathbf{C}_M[i]), \K^{\rm{p}}_{\rm{e}}\tilde{F}(\mathbf{C}_T[j])).
	 \end{align*}
	 where the second isomorphism obtained from the adjoint pair $(k^v_{\la}, k^v)$, for the third $\K^{\rm{p}}_{\rm{e}}F:\K^{-, \rm{p}}(\mmod R)\rt \K^{-, \rm{p}}(\mmod R')$ is the induced functor, introduced in Construction \ref{Costruction 2}, which is itself an equivalence by Proposition \ref{Proposition 4.9}(3),  for the forth we again use the adjoint pair $(k^v_{\la}, k^v)$ but this turn relative to $\rm{rep}(\CQ, R')$, and finally we use Lemma \ref{Lemma 6.3}. 
\end{proof}
Let $\mathbb{M}_n(R)$ be the set of all $n\times n$ square matrices with coefficients in $R$ for $n\in \mathbb{N}$. $\mathbb{M}_n(R)$ is a ring with respect to the
usual matrix addition and multiplication. The subset
$$\mathbb{T}_n(R)=
\left[ \begin{array}{cccc}

R & 0 & \cdots & 0 \\
R & R & \cdots & 0 \\
\vdots & \vdots & \ddots & \vdots \\
R & R & \cdots & R \\
\end{array} \right]     
$$
of $\mathbb{M}_n(R)$ consisting of all triangular matrices $[a_{ij}]$ in $\mathbb{M}_n(R)$ with zeros
over the main diagonal is a subring of $\mathbb{M}_n(R)$. It is well known for the   quiver 
$$\xymatrix{ A_n : 1 \ar[r] & 2 \ar[r] &3 \ar[r]  & \cdots \ar[r]& n,  }$$  there is an isomorphism of rings $RA_n\cong \mathbb{T}_n(R)$.
We specialize our results
 on the quiver $A_n$, then:
\begin{corollary}\label{corollary 6.5}
Let $F:\mmod R\rt \mmod R'$ be the same as the above and $n \in \mathbb{N}$. Then, $T_n(R)$ and $T_n(R')$ are singularly equivalent.	
\end{corollary}
\subsection{Triangular 	Matrix Rings}
Throughout  this subsection let $R$ and $S$ be two given right notherian rings and  $S$-$R$-bimodule ${}_SM_R$ with $M_R$ is a finitely generated $R$-module. The triangular matrix ring
$ \tiny {T=\left[\begin{array}{ll} R & 0 \\ M & S \end{array} \right]}$ has as its elements matrices  $ \tiny {\left[\begin{array}{ll} r & 0 \\ m & s \end{array} \right]}$ where $r \in R, \  s \in S$ and $m \in M$ with addition defined coordinate wise and multiplication given by $ \tiny {\left[\begin{array}{ll} r & 0 \\ m & s \end{array} \right]} \tiny {\left[\begin{array}{ll} r' & 0 \\ m' & s' \end{array} \right]}=  \tiny {\left[\begin{array}{ll} rr' & 0 \\ mr'+sm' & ss' \end{array} \right]}.$ It is not difficult to see that $T$ is right notherian as well. It is well-known \cite{Gr} that $\mmod T$ is equivalent to the category $\CT$ consisting of all triples $(X, Y)_f$ where $ X \in \mmod R$, $Y \in \mmod S$ and $f:Y\otimes_SM_R \rt X$ is a $R$-linear map. A morphism $(X_1, Y_1)_f\rt (X_2, Y_2)_{f'}$ is a pair $(\phi_1, \phi_2)$ where $\phi_1:X_1\rt X_2$, $\phi_2:Y_1 \rt Y_2$, such that there is the following commutative diagram 
	\[\xymatrix{ Y_1 \otimes_SM_R \ar[r]^<<<<<{f} \ar[d]^{\phi_2\otimes _S M_R} & X_1 \ar[d]^{\phi_1} \\ Y_2\otimes_S M_R \ar[r]^<<<<<{f'} & X_2. }\]
Due to this equivalence we identify $\mmod T$ with $\CT.$ Analogous to representations of quivers, we have two evaluation functors $e^1:\mmod T\rt \mmod R$ and $e^2:\mmod T\rt \mmod S$ in the following way: for every $T$-module $(X, Y)_f$,  $e^1((X, Y)_f)=X$ and $e^2((X, Y)_f)=Y$. In this setting, the evaluation functors also have left adjoints. They are explicitly described as follows: For any $X \in \mmod R$, $e^1_{\la}(X)=(X, 0)_0$, and for any $Y \in \mmod S$,  $e^2_{\la}(Y)=(Y\otimes_S M_R, Y)_1$, here ``$1$'' denotes the identity morphism $\rm{Id}_{Y\otimes_S M_R}$. But in this case the left adjoint are not exact in general. More precisely, $e^1_{\la}$ is exact but the latter one might be not  in general. \\
Let $\CM(T)$ be the subcategory of $\mmod T$ formed by all $T$-modules which are isomorphic to a direct sum $e^1_{\la}(M)\oplus e^2_{\la}(N)$ for some $M \in \mmod R$ and $N \in \mmod S.$ One can see easily $\CM(T)$ formed by all objects $(X, Y)_f$ such that $f:Y\otimes_SM_R\rt X$ is a split monomorphism.
\begin{lemma}
	Let $\CM(T)$ be  defined as in the above. The subcategory $\CM(T)$ is contravariantly finite in $\mmod T$, and the global  $\CM(T)$-dimension of $T$ is at most one.
\end{lemma}
\begin{proof}
	The proof in general is the same as Lemma \ref{lemma 6.1}. In this setting we also have a short exact sequence for any $(X, Y)_f$, similar to $(\dagger)$  in subsection \ref{Subsection 3.1},   as the following
	$$0 \rt (Y\otimes_SM_R, 0)_0\rt (X, 0)_0\oplus (Y\otimes_S M_R, Y)_1\rt (X, Y)_f\rt 0$$
	Analogous to $(\dagger)$ the above short exact sequence is essential to prove the lemma.
\end{proof}
An immediate consequence of lemma above in conjunction with Theorem \ref{Theorem 3.10} is the following:
\begin{theorem}\label{theorem 6.7}
	let $T$ be a triangular matrix ring  $ \tiny {T=\left[\begin{array}{ll} R & 0 \\ M & S \end{array} \right]}$ and $\CM(T)$ as defined above. Then, there is the following equivalence of triangulated categories
	$$\D_{\rm{sg}}(T)\simeq \frac{\K^{-, \rm{p}}(\CM(T))}{\K^{\rm{b}}_{\rm{ac}}(\CM(T))}.$$
	\end{theorem}
As an application of the above description of the singularity categories of triangular matrix rings, in the next result, we will show how a certain  stable equivalence can be lifted to a singular equivalence. 
\begin{proposition}
let $T$ be a triangular matrix ring  $ \tiny {T=\left[\begin{array}{ll} R & 0 \\ M & S \end{array} \right]}$. Assume the functor $-\otimes_SM_R:\mmod S\rt \mmod R$ induces
 an equivalence between the stable categories $\underline{\rm{mod}}\mbox{-}S\simeq \underline{\rm{mod}}\mbox{-}R$.	Necessarily, the functor must preserve the projective modules. Then,  the following triangle equivalence of the singularity categories occurs:

	$$\D_{\rm{sg}}(\tiny {\left[\begin{array}{ll} R & 0 \\ M & S \end{array} \right]})\simeq \D_{\rm{sg}}(\tiny {\left[\begin{array}{ll} R & 0 \\ R & R \end{array} \right]}).$$ 
\end{proposition}
\begin{proof}
 Setting  $\tiny {T= \left[\begin{array}{ll} R & 0 \\ M & S \end{array} \right]}$ and $\tiny {T_2(R)= \left[\begin{array}{ll} R & 0 \\ R & R \end{array} \right]}$.  Let $\CM(T)$ and $\CM(T_2(R))$ be  defined as in the above for the corresponding triangular matrix rings. Indeed, the subcategory $\CM(T_2(R))$ is nothing else than the subcategory $\CM_{R}(A_2)$ of the representations over the quiver $A_2.$ For the  connivance, we prefer to work with $\CM_{R}(A_2)$ instead of $\CM(T_2(R))$. To get the result, we use Theorem \ref{Coroallry 5.7} (3). For this we need to define a functor $F:\CM(T)\rt \CM_{R}(A_2)$ satisfying the required conditions in the theorem and to be an restriction of an exact functor between some quasi-resolving subcategories. Let $\CS(T)$ be the subcategory of $\mmod T$ formed by all modules $(X, Y)_f$ such that $Y\otimes_SM_R\rt X$ is a monomorphism. The subcategory $\CS(R)$ of $\mmod R A_2$ is defined similarly, indeed, it is the category of all (submodules) monomorphisms in $\mmod R,$ in the sense of Ringel and Schmidmeier.  One can see both subcategories are quasi-resolving and contain the subcategories $\CM(T)$ and $\CM_{R}(A_2)$, respectively.   Now we define the functor $G:\CS(T)\rt \CS(R)$ in the following way. Let $(A, B)_f$ in $\CS(T)$ be given.  Then $B\otimes_S M_R \st{f}\rt A$ is a  monomorphism. But this means that the  monomorphism $f$ lies as an object  in $\CS(R)$. Define $G((A, B)_f):=f$, and on morphisms $G$ is defined in an obvious way.  Now we show that the functor $G$ already defined preserves  short exact sequences in $\CS(T)$. Let $0 \rt (A^0, B^0)_{f^0}\rt (A^1, B^1)_{f^1}\rt (A^2, B^2)_{f^2}\rt 0$ be a short exact with all terms in $\CS(T)$. Then it induces the following commutative diagram with exact rows in $\mmod R$
	\[\xymatrix{ & B^0\otimes_S M_R \ar[d]_{f^0} \ar[r] &  B^1\otimes_S M_R \ar[d]^{f^1} \ar[r] &  B^2\otimes_S M_R \ar[d]^{f^2} \ar[r] &0  \\ 0  \ar[r] & A^0\ar[r] & A^1\ar[r] & A^2\ar[r] & 0.}\]
	Since $f^0$ is monomorphism we get the sequence in the top is also a short exact sequence. Hence the above commutative diagram gives us the following short exact sequence in $\mmod R A_2$ {\footnotesize  \[ \xymatrix@R-2pc {  &  ~ B^0\otimes_S M_R\ar[dd]^{f^0}~   & B^1\otimes_R M_S\ar[dd]^{f^1}~  & B^2\otimes_R M_S\ar[dd]^{f^2} \\   0 \ar[r] &  _{ \ \ \ \ } \ar[r]  &_{\ \ \ \ \ } \ar[r]_{\ \ \ \ \ }&  _{\ \ \ \ \ }\ar[r] & 0, \\ & A^0& A^1 & A^2 }\]} where it is actually the image of the given short exact sequence in $\CM(T)$ under the functor $G$. So we are done the claim.	We know by \cite[Theorem 3.1]{HV} any projective module in $\mmod T$ can be written up to isomorphism as $(P, 0)_0\oplus (Q\otimes_SM_R, Q)_1$ for some $P\in \rm{prj}\mbox{-}R$ and $Q \in \rm{prj}\mbox{-}S.$ Using the fact which the functors $-\otimes_SM_R$  preserves the projective modules,  and  in view of the characterization of projective modules over triangular matrix rings, we can  verify that the functor $G$ also preserves the projective modules. Set $F=G\mid_{\CM(T)}$. The essential image of the restricted functor $F$ is clearly in $\CM_{R}(A_2)$. To complete the proof, we must  show
	that the induced functor $\underline{F}:\underline{\CM(T)}\rt \underline{\CM_{R}(A_2)}$ 
	is an equivalence. As any module in $\CM(T)$ (and also similarly  in $\CM_{R}(A_2)$ for the relevant cases) is isomorphic to a finite direct sum of modules in the forms $(X, 0)_0$ or $(Y\otimes_SM_R, Y)_1$ where $X \in \mmod R$ and $Y \in \mmod S$, hence to prove the claim  it is enough to concentrate on  such modules. For density, let $(0\rt Z)$ and $(W\st{1}\rt W)$ in $\CM_{R}(A_2)$ are given. Since $-\otimes_{S}M_R$ induces an equivalence from $\underline{\rm{mod}}\mbox{-}S$ to $\underline{\rm{mod}}\mbox{-}R$, then there exists projective modules $P$ and $Q$ in $\rm{prj}\mbox{-}R$ and $U \in \mmod S$ with isomorphism $W\oplus P\simeq (U\otimes_{S}M_R)\oplus Q$ in $\mmod R.$ With the isomorphism we can get the $(W\st{1}\rt W)\oplus (P\st{1}\rt P)\simeq F((U\otimes_{S}M_{R}, U)_1)\oplus (Q\st{1}\rt Q).$ Hence $F((U\otimes_{S}M_{R}, U)_1)\simeq (W\st{1}\rt W)$ in $\underline{\CM_{R}(A_2)}$, so the result follows. The case of $(0\rt Z)$ is easy. For fullness and faithfulness, assume that $D$ and $E$ are in $\CM(T)$. We must prove the induced group homomorphism $G:\underline{\Hom}_T(D, E)\rt \underline{\Hom}_{T_2(R)}(F(D), F(E))$ is bijective. To prove, it is enough to consider the following cases. $(1)$ If $D=(X, 0)_0$ and $E=(Y, 0)_0$, where $X, Y \in \mmod R,$ then it is clear. $(2)$ If $D=(Y\otimes_{S} M_{R}, Y)_1$ and $E=(X, 0)_0$, where $X \in \mmod R$ and $Y \in \mmod S$, then since the both
	Hom-spaces are zero, then this case is also clear. $(3)$ If $D=(X, 0)_0$ and $E=(Y\otimes_{S}M_{R}, Y)_1$, then both the Hom-spaces are isomorphic to the Hom-space $\underline{\Hom}_{R}(X, Y\otimes_{S}M_{R})$ by the relevant adjoint pairs. $(4)$ If $D=(Y\otimes_{S}M_{R}, Y)_1$ and $E=(Z\otimes_{S}M_{R}, Z)_1$, then by using the relevant adjoint pairs and the equivalence induced by $-\otimes_{S}M_{R}$, we have:	  
	\begin{align*}
	\underline{\Hom}_{T_2(R)}(F(D), F(E)) &\simeq \underline{\rm{Hom}}_{R}(Y\otimes_{S}M_{R}, Z\otimes_{S}M_{R})\\
	&\simeq\underline{\rm{Hom}}_{S}(Y, Z) \\
	&\simeq \underline{\rm{Hom}}_{T}(D, E).
	\end{align*}
	Now Theorem \ref{Coroallry 5.7} (3)   completes the proof.
\end{proof}

Stable equivalences of Morita type are a special case of the above proposition. Hence we have the following corollary.

\begin{corollary}
	Let $\La$ and $\La'$ be two Artin algebras of stable equivalence of Morita type. Assume pair of bimodules $({}_{\La'}M_{\La}, {}_{\La}N_{\La'})$ satisfies the required condition,  that is,
	  there exists a pair of bimodules ${}_{\La'}M_{\La}$ and ${}_{\La}N_{\La'}$ such that
	\begin{itemize}
		\item [$(1)$] $M$ and $N$ are projective as left and right modules, respectively;
		\item[$(2)$]  $M\otimes_{\La}N\simeq \La'\oplus P$   as $\La'\mbox{-}\La'$-bimodules for some projective $\La'\mbox{-}\La'$-bimodule $P$, 	and $N \otimes_{\La'}M\simeq \La\oplus Q $  as $\La\mbox{-}\La$-bimodules for some projective $\La\mbox{-}\La$-bimodule $Q$.
	\end{itemize} Then,  the following triangle equivalences of the singularity categories occur:
\begin{itemize}
	\item [$(i)$]	$$\D_{\rm{sg}}(\tiny {\left[\begin{array}{ll} \La & 0 \\ {}_{\La'}M_{\La} & \La' \end{array} \right]})\simeq \D_{\rm{sg}}(\tiny {\left[\begin{array}{ll} \La & 0 \\ \La & \La \end{array} \right]}).$$ 
\item[$(ii)$]
	$$\D_{\rm{sg}}(\tiny {\left[\begin{array}{ll} \La' & 0 \\ {}_{\La}N_{\La'} & \La \end{array} \right]})\simeq \D_{\rm{sg}}(\tiny {\left[\begin{array}{ll} \La' & 0 \\ \La' & \La' \end{array} \right]}).$$ 	
\end{itemize}	
In particular, by Corollary \ref{Coroallry 5.7}, we have $\D_{\rm{sg}}(T_2(\La))\simeq\D_{\rm{sg}}(T_2(\La'))$. Hence, 	
	$$\D_{\rm{sg}}(\tiny {\left[\begin{array}{ll} \La & 0 \\ {}_{\La'}M_{\La} & \La' \end{array} \right]})\simeq \D_{\rm{sg}}(\tiny {\left[\begin{array}{ll} \La' & 0 \\ {}_{\La}N_{\La'} & \La \end{array} \right]}).$$ 
	
\end{corollary}

\subsection{Trivial extension and tensor rings}
First we recall some facts on modules over trivial extension rings. Throughout  this subsection,  let as usual  $R$ be a right notherian ring and let ${}_{R}M_{R}$ be a $R$-$R$-bimodule and $M_R$ is finitely generated. The trivial extension of $R$ by $M$ is a ring, denoted by $R\ltimes M$, with the underlying abelian group $R\oplus M$ and the multiplication: $(r, m)(r', m')=(rr', rm'+mr')$ for $r, r'$ in $R$ and $m, m'$ in $M.$ The actions involved in  the definition of the multiplication can be viewed in a natural way.  Then by our assumption $R\ltimes M$ is also a right notherian ring.  For further information on
$R\ltimes M$, we refer the reader to \cite{FGR}. Let us mention here a triangular matrix $ \tiny {\left[\begin{array}{ll} R & 0 \\ M & S \end{array} \right]}$ is a trivial extension of the product ring $ R \times S$ by  $M$. 
In this subsection we give a generalization of Theorem \ref{theorem 6.7} in terms of trivial extension rings. Similar to the triangular matrix rings we can identify a right module over $R\ltimes M$ with a pair $(X, \sigma)$, where $X$ is a right $R$-module and $\sigma:X\otimes_RM_{R}\rt X$ is a morphism of right $R$-modules with the property $\sigma \circ (\sigma\otimes_RM_{R})=0$. And a morphism $(X, \sigma)\rt (Y, \eta)$ of $T$-modules is just a morphism $f:M\rt N$ of $R$-modules  satisfying $\eta \circ (f\otimes_RM_{R})=f\circ \eta.$ We write $f:(M, \sigma)\rt (N, \eta)$. For any $R\ltimes M$-module $(X, \sigma)$, similar to the ones for triangular matrix rings and path rings, we have the following short exact sequence in $\mmod R\ltimes M,$ for any module $(X, \sigma)$
$$
0 \rt (X\otimes_{R}M_{R}, \sigma\otimes_{R}M)\st{[\sigma~~-1]^{t}}\rt (X\oplus (X\otimes_{R}M_{
R}), \tiny {\left[\begin{array}{ll} 0 & 0 \\ \rm{Id}_{X\otimes M} & 0 \end{array} \right]}  )\st{[1~~\sigma]}\rt (X, \sigma)\rt 0.
$$
See also \cite[Lemma 3.1]{C2} for the above short exact sequence.\\
For $R$-$R$-bimodule $M$, write $M^{\otimes_R0}=R$ and $M^{\otimes_R(j+1)}=M\otimes_RM^{\otimes_Rj}$ for $j\geqslant 0.$ We say that $M$ is {\it nilpotent}, if $M^{\otimes_R(n+1)}=0$ for some $n\geqslant0.$ From now on, assume that $M$ is nilpotent and for $n\geqslant0$, $M^{\otimes_R(n+1)}=0$. For a given $R\ltimes M$-module $(X, \sigma)$, denote by $\epsilon_1$ the above short exact sequence. Replacing $(X, \sigma)$ with $(X\otimes_RM_R, \sigma\otimes_RM)$ the starting term of $\epsilon_0$, we get
$$0 \rt (X\otimes M^{\otimes_R2}, \sigma \otimes M^{\otimes_{R}2})\rt ((X\otimes M)\oplus(X\otimes M^{\otimes_R 2}), \tiny {\left[\begin{array}{ll} 0 & 0 \\ \rm{Id}_{X\otimes M^{\otimes_R 2}} & 0 \end{array} \right]}))\rt (X\otimes M, \sigma\otimes_R M)\rt 0.$$ 
Inductively, by replacing with starting term of the short exact sequence $\epsilon_i$ we get the short exact sequence $\epsilon_{i+1}$
{\tiny \begin{equation*} 0 \rt (X\otimes M^{\otimes(i+1)}, \sigma \otimes M^{\otimes(i+1)})\rt ((X\otimes M^{\otimes i})\oplus(X\otimes M^{\otimes(i+1)}), \tiny {\left[\begin{array}{ll} 0 & 0 \\ \rm{Id}_{X\otimes M^{\otimes(i+1)}} & 0 \end{array} \right]}))\rt (X\otimes M^{\otimes i}, \sigma\otimes M^{\otimes i})\rt 0.\end{equation*}}
For brevity, in the above we drop the subscript ``$R$''. In  the rest, we shall also drop in case of need.  Denote by $\CM(R\ltimes M)$ the subcategory of $\mmod R\ltimes M$ consisting of all modules isomorphic to $(X\oplus (X\otimes_{R}M_{R}), \tiny {\left[\begin{array}{ll} 0 & 0 \\ \rm{Id}_{X\otimes M} & 0 \end{array} \right]}  )$ for some $X$ in $\mmod R.$ The middle term of each short exact sequence $\epsilon_i$ is in $\CM(R\ltimes M)$. Since $M^{\otimes_{R}(n+1)}=0$ we observe that the starting term in the $\epsilon_n$ has to be in $\CM(R\ltimes M)$. Hence by splicing together all the short exact sequences $\epsilon_1, \cdots, \epsilon_n$, we get an exact sequence in $\mmod R\ltimes M$ with all terms in $\CM(R\ltimes M)$ except possibly the end term

$$0 \rt (X\otimes M^{\otimes_n}, 0)\rt((X\otimes M^{\otimes n-1})\oplus(X\otimes M^{\otimes(n)}), H_n)\rt \cdots \rt (X\oplus (X\otimes_{R}M_{\La}), H_1  )\rt (X, \sigma)\rt0,  $$
where $\tiny {H_i=\left[\begin{array}{ll} 0 & 0 \\ \rm{Id}_{X\otimes M^{\otimes(n)}} & 0 \end{array} \right]})$, for $i=1,\cdots, n$. Note that the the starting term $(X\otimes M^{\otimes_n}, 0)$ is in $\CM(R\ltimes M)$ as it is indeed equal to $ ((X\otimes M^{\otimes n})\oplus(X\otimes M^{\otimes(n+1)}), \tiny {\left[\begin{array}{ll} 0 & 0 \\ \rm{Id}_{X\otimes M^{\otimes(n+1)}} & 0 \end{array} \right]})$.
 
\begin{lemma}
	The subcategory $\CM(R\ltimes M)$ is contravariantly finite.
\end{lemma}
\begin{proof}
	Let $(X, \sigma)$ in $\mmod R\ltimes M$ be given. We shall show the the epimorphisem $(X\oplus (X\otimes_{R}M_{\La}), \tiny {\left[\begin{array}{ll} 0 & 0 \\ \rm{Id}_{X\otimes M} & 0 \end{array} \right]}  )\st{[1~~\sigma]}\rt (X, \sigma)\rt 0,$ appeared in the short exact sequence introduced in the above, works as a right $\CM(R\ltimes M)$-approximation. Assume $[f, g]:(Y\oplus(Y\otimes_{R}M_R))\rt (X, \sigma)$, where $f:Y\rt X$ and $g: Y\otimes_{R}M_R\rt X$, is given. It is a direct verification to check that $\tiny {\left[\begin{array}{ll} f & 0 \\ 0 & f\otimes M \end{array} \right]}: Y\oplus (Y\otimes_{R}M_R)\rt X\oplus (X\otimes_RM_R)$ is a morphism in $\mmod R\ltimes M$ and factors $[f, g]$ through $[1, \sigma]$. We are done.
\end{proof}
Now the observations provided in the above enables us to use our result to give a description for $\D_{\rm{sg}}(R\ltimes M)$.
\begin{theorem}\label{Theroem 6.10}
Let $R$ be a right notherian ring and let $R$-$R$-bimodule $M$ be nilpotent and as right $R$-module is finitely generated. Then, there is the following equivalence of triangulated categories
$$\D_{\rm{sg}}(R\ltimes M)\simeq \frac{\K^{-, \rm{p}}(\CM(R\ltimes M))}{\K^{\rm{b}}_{\rm{ac}}(\CM(R\ltimes M))}.$$
\end{theorem}
\begin{proof}
	The theorem is an immediate consequence of Theorem \ref{Theorem 3.10} and in conjunction with facts provided in the above. Just note that by \cite[Corollary 6.1]{FGR},  $\CM(R\ltimes M)$ contains the projective modules.
\end{proof}
With nilpotent $R$-$R$-bimodule $M$ we can construct another right notherian ring called {\it tensor ring}, that is, $T_R(M)= \oplus_{i=0}^{\infty}M^{\otimes_Ri}$. As mentioned in \cite[Section 3]{CL}, one can identify right modules in $\mmod T_{R}(M)$ with the representations of the endofunctor $-\otimes_RM_R:\mmod R\rt \mmod R$. By a representation of $-\otimes_{R} M$, we mean a pair $(X, u)$ with $X$ in $\mmod R$ and $u:X\otimes_{R}M\rt X.$ In fact, the  modules over $R\ltimes M$ can be considered as the representations of $-\otimes_RM_R$.  A morphism between two representations is the same as one in $\mmod R\ltimes M.$ Such identifications allow us to consider $\mmod R\ltimes M$ as an  abelian subcategory of $\mmod T_R(M)$. For each module $X$ in $\mmod R$, following \cite[Section2]{CL} we define $\rm{Ind}(X)=\oplus^{n}_{i=0}X\otimes_{R}M^{\otimes_{R}i}$, remember $n$ is an integer with $M^{\otimes_{R}(n+1)}=0$, and moreover, a morphism $c_X:\rm{Ind}(X)\otimes_{R}M_R\rt \rm{Ind}(X)$ such its restriction to $X\otimes^{\otimes_{R}i}$, $i>0$, is the inclusion into $\rm{Ind}(X)$. This defines
 the $T_R(M)$-module $(\rm{ind}(X), c_X)$. The assignment $(\rm{Ind}(X), c_X)$ to each $X$ gives rise to a functor $\rm{Ind}:\mmod R\rt \mmod T_R(M)$. As shown in \cite[lemma 2.1]{CL} it is a left adjoint of the forgetful functor $U:\mmod T_R(M)\rt \mmod R$ by sending $(X, u)$ to the underlying module $X.$ Due to this adjoint pair one can see that any projective module in $\mmod T_R(M)$ is isomorphic to $(\rm{Ind}(P), c_P)$ for some $P$ in $\rm{prj}\mbox{-}R.$ Furthermore, for any module $(X, u)$ in $\mmod T_R(M)$, there is an exact sequence in $\mmod T_R(M)$ as the following 
$$0 \rt (\rm{Ind}(X\otimes_{R}M), c_{X\otimes_{R}M})\rt (\rm{Ind}(X), c_X)\rt (X, u)\rt 0.$$
We refer the reader to \cite[Section 2]{CL} for more details of the construction  of the above sequence. Denote by $\CM(T_R(M))$ the subcategory of $\mmod T_R(M)$ formed by all modules which are isomorphic to $(\rm{Ind}(X), c_X)$ for some $X$ in $\mmod R.$ By using the above sequence one can prove that $\CM(T_R(M))$ is a contravariantly finite subcategory in $\mmod T_R(M)$ and the global dimension relative to it is at most one. We leave the proofs to the reader as they can be proved in the similar

 cases in the preceding subsections. Consequently, the subcategory $\CM(T_R(M))$ satisfies all the required conditions of Theorem \ref{Theorem 3.10}. Therefore, we have the following result.
\begin{theorem}
	Let $R$ and $M$ be the same as in Theorem \ref{Theroem 6.10}. Then, there is the following equivalence of triangulated categories
	$$\D_{\rm{sg}}(T_R(M))\simeq \frac{\K^{-, \rm{p}}(\CM(T_R(M)))}{\K^{\rm{b}}_{\rm{ac}}(\CM(T_R(M)))}.$$
\end{theorem}
One important advantage of the  above theorem, and also the similar results in Theorems \ref{Theroem 6.10}, \ref{theorem 6.7} and \ref{Theorem 6.3}, is to establish a close connection between the singularity category of $\D_{\rm{sg}}(T_R(M))$ and the singularity category $\D_{\rm{sg}}(R)$ of the base ring $R.$
\section{Acknowledgments}
 I would like to thank my wife  for her support and patience during  this work.

\end{document}